\documentclass[a4paper]{amsart}
\usepackage{relsize}
\usepackage{epic}

\usepackage{amsmath, amssymb}
\usepackage{amsthm}
\usepackage{latexsym}

\usepackage{fullpage}
\usepackage{xcolor}
\usepackage{MnSymbol} 

\usepackage{verbatim} %for commenting out large sections
\usepackage{longtable} %for tables that go over 2 pages
\usepackage{enumerate} %for changing items to (a)

\usepackage{fullpage}
\usepackage{longtable} %for tables that go over 2 pages

\usepackage{makecell}

\newtheorem{theorem}{Theorem}[section]
\newtheorem{lemma}[theorem]{Lemma}
\newtheorem{proposition}[theorem]{Proposition}
\newtheorem{corollary}[theorem]{Corollary}

\newtheorem{definition}[theorem]{Definition}
\newtheorem{notation}[theorem]{Notation}
\newtheorem{remark}[theorem]{Remark}

\newtheorem{assumption}[theorem]{Assumption}

\newtheorem*{theorem*}{Theorem}
\newtheorem*{conjecture*}{Conjecture}

\newtheorem{theoremA}{Theorem}

\newtheorem{theoremB}{Theorem}

\let\olddefinition\definition
\renewcommand{\definition}{\olddefinition\normalfont} % These two lines mean that definitions won't be in italics! 

\let\oldnotation\notation
\renewcommand{\notation}{\oldnotation\normalfont} % These two lines mean that notation won't be in italics! 

\let\oldremark\remark
\renewcommand{\remark}{\oldremark\normalfont} % These two lines mean that definitions won't be in italics! 

\usepackage{enumitem}
\setlist[enumerate]{label=(\alph*), itemsep= .7ex,leftmargin=4.4ex, parsep = .5ex, topsep = 1ex }

\newcommand{\tn}[1]{{\textnormal{#1}}}

\newcommand{\mL}{{\mathcal{L}}}
\newcommand{\mP}{{\mathcal{P}}}
\newcommand{\mD}{{\mathcal{D}}}
\newcommand{\mG}{{\mathcal{G}}}

\newcommand{\F}{{\mathbb{F}}}

\newcommand{\N}{{\mathbb{N}}}
\newcommand{\Q}{{\mathbb{Q}}}

\newcommand{\Z}{{\mathbb{Z}}}

\newcommand{\bG}{{\textbf{G}}}
\newcommand{\bT}{{\textbf{T}}}

\newcommand{\wtG}{{\tilde{\textbf{G}}}}

\newcommand{\wtT}{{\tilde{\textbf{T}}}}
\newcommand{\wtGF}{{\tilde{\textbf{G}}^F}}

\newcommand{\lam}{{\lambda}}

\newcommand{\ra}{\to}

\newcommand{\IBr}{{\operatorname{IBr}}}
\newcommand{\Irr}{{\operatorname{Irr}}}
\newcommand{\Aut}{{\operatorname{Aut}}}
\newcommand{\Out}{{\operatorname{Out}}}
\newcommand{\Inn}{{\operatorname{Inn}}}

\begin{document}

\title{Fake Galois actions }
\date{\today}
\author{Niamh Farrell}
\address{Fachbereich Mathematik, TU Kaiserslautern, Postfach 3049, 67653 Kaiserslautern, Germany}
\email{farrell@mathematik.uni-kl.de}
\author{Lucas Ruhstorfer}
\address{Bergische Universit\"{a}t Wuppertal, Gaussstrasse 20, 42119 Wuppertal, Germany}
\email{ruhstorfer@uni-wuppertal.de}



\begin{abstract}
	We prove that for all non-abelian finite simple groups $S$, there exists a fake $m$th Galois action on $\IBr(X)$ with respect to $X \lhd X \rtimes \Aut(X)$, where $X$ is the universal covering group of $S$ and $m$ is any non-negative integer coprime to the order of $X$. This is one of the two inductive conditions needed to prove an $\ell$-modular analogue of the Glauberman--Isaacs correspondence. 
\end{abstract}

\thanks{The first author gratefully acknowledges financial support received from a London Mathematical Society 150th Anniversary Postdoctoral Mobility Grant, and from the DFG project SFB-TRR 195. \\ \indent This material is based upon work supported by the National Science Foundation under Grant No. DMS-1440140 while the second named author was in residence at the Mathematical Sciences Research Institute in Berkeley, California, during the Spring 2018 semester. The second author's research was conducted in the framework of the research training group GRK 2240: Algebro-geometric
Methods in Algebra, Arithmetic and Topology, which is funded by the DFG} 

	\maketitle

% ----------------------------------------------------------------
% ----------------------------------------------------------------
\section{Introduction}
% ----------------------------------------------------------------
% ----------------------------------------------------------------

Let $G$ and $A$ be finite groups such that $A$ acts on $G$ via group automorphisms. Then the action of $A$ on $G$ induces an action on the set of complex irreducible characters of $G$. Deep work of Glauberman in the 1960's revealed an astonishing correspondence between the $A$-invariant characters of $G$, and the characters of the fixed point subgroup $\mathrm{C}_{G}(A)$, when $A$ and $G$ have coprime orders and $A$ is solvable. This work was extended by Isaacs in the 1970's and their combined results became known as the Glauberman--Isaacs correspondence. This is a very important result in the representation theory of finite groups, and it has been used in a fundamental way in solving many problems relating to group actions, as well as in work on some of the famous open conjectures in modular representation theory such as the Alperin weight conjecture. 

\begin{theorem*}[Glauberman--Isaacs Correspondence]
	\label{thrm:GIC}
	Let $G$ and $A$ be finite groups. Suppose that $A$ acts on $G$ via group automorphisms and suppose that $(|A|, |G|) = 1$. Then there exists a canonical bijection between the set of $A$-invariant irreducible characters of $G$ and the set of irreducible characters of the fixed point subgroup $\mathrm{C}_G(A)$.
\end{theorem*}

We are concerned with an $\ell$-modular analogue of the Glauberman--Isaacs correspondence. Let $\ell$ be a prime and suppose that $A$ and $G$ are as in the Theorem above. Then the action of $A$ on $G$ yields a natural action on the $\ell$-Brauer characters of $G$. In \cite[Problem 4]{Nav}, Navarro posed the following question: does there exist a bijection between the set of $A$-invariant irreducible $\ell$-Brauer characters of $G$ and the set of irreducible $\ell$-Brauer characters of $\mathrm{C}_G(A)$? The answer is known in some specific cases, for example the answer is yes if $G$ is $\ell$-solvable \cite{U}, but it is not known in general whether such a correspondence always exists. Sp\"{a}th and Vallejo Rodr\'iguez proved a reduction theorem for this question in \cite{S/V}. For this they introduced a set of conditions on finite simple groups which they called the \textit{inductive Brauer-Glauberman (IBG) condition} \cite[Definition 6.1]{S/V}, and proved the following.

\begin{theorem*}[{\cite[Theorem A]{S/V}}]
	\label{thrm:SVR}
	Let $G$ and $A$ be finite groups. Suppose that $A$ acts on $G$ with $(|A|, |G|) = 1$. Suppose that all finite non-abelian simple groups involved in $G$ satisfy the IBG condition. Then the number of $A$-invariant irreducible $\ell$-Brauer characters of $G$ is equal to the number of irreducible $\ell$-Brauer characters of $\mathrm{C}_{G}(A)$.
\end{theorem*}

The IBG condition is known to hold for finite simple groups of Lie type defined over a field of characteristic $\ell$, and for finite simple groups not of Lie type \cite{N/S/T}. In order to prove Navarro's conjecture it therefore remains to show that the IBG condition holds for all finite simple groups of Lie type in non-defining characteristic. This condition has two parts. The first part is a stronger, suitably equivariant version of a modular Glauberman--Isaacs correspondence. The second part is the existence of what Sp\"{a}th and Vallejo Rodr\'iguez call a \textit{fake $m$th Galois action}, see Definition \ref{def:fakegalaction}. Galois automorphisms act on ordinary characters and change their Clifford theory. A fake Galois action is a bijective map on the set of irreducible Brauer characters changing the Clifford theory in an analogous way. We are concerned here with this part of the condition and prove the following result. 

\begin{theoremA}
	Let $S$ be a non-abelian simple group and let $X$ be the universal covering group of $S$. Then for all non-negative integers $m$ such that $(|X|, m ) = 1$, there exists a fake $m$th Galois action on $\IBr(X)$ with respect to $X \lhd X \rtimes \Aut(X)$. 
\end{theoremA}

The organisation of this paper is as follows. In Section~\ref{sec:notationanddefs} we recall the most important definitions and notation. We also state some general results about projective representations which might be of independent interest. In Section \ref{sec:TypeA} we prove the existence of fake Galois actions for groups of Lie type A. The main idea is to use the ordinary Galois action on complex characters and the fact that the $\ell$-decomposition matrix of $\mathrm{SL}_n(q)$ is unitriangular. In Section \ref{sec:notofTypeA} we deal with all groups of Lie type not of type A. As the outer automorphism group of D$_4(q)$ is more complicated than in the other cases, it often has to be dealt with using separate arguments. The construction of fake Galois actions for exceptional covering groups of simple finite groups of Lie type with non-cyclic outer automorphism groups is in Section \ref{sec:except}. Finally, the proof of the main theorem is in Section \ref{sec:maintheorem}.

\vspace{2ex} 
\textbf{Acknowledgements}
We are extremely grateful to  Professor Britta Sp\"{a}th for her support and encouragement on this project. The first author would like to thank the Bergische Universit\"{a}t Wuppertal for generously hosting her for a period of six months during which much of this article was written. We also thank Professor Gunter Malle for his careful reading of an earlier manuscript and for his helpful comments.

% ----------------------------------------------------------------
% ----------------------------------------------------------------
\section{Notation, definitions and preliminary results}\label{sec:notationanddefs}
% ----------------------------------------------------------------
% ----------------------------------------------------------------

\subsection{Notation}
Let $\ell$ be a prime and let $\F = \overline \F_{\ell}$ be an algebraic closure of the finite field of $\ell$ elements. Let $G$ be a finite group and let $N \lhd G$. We let $\Irr(G)$ denote the set of complex irreducible characters of $G$ and $\IBr(G)$ denote the set of $\ell$-modular Brauer characters of $G$ (see \cite[Chapter 2]{N}). For $\chi \in \Irr(G)$, $\chi_N$ denotes the restriction of $\chi$ to $N$ and $\Irr(N \mid \chi)$ denotes the set of irreducible constituents of $\chi_N$. For $\theta \in \Irr(N)$, $\theta^G$ denotes the induced character of $G$ and $\Irr(G \mid \theta)$ denotes the set of irreducible constituents of $\theta^G$. The notation for irreducible Brauer characters is analogous (see \cite[Corollary 8.7]{N}). For $\theta$ a complex or irreducible character of $G$, let $G_{\theta}$ denote the stabiliser of $\theta$ in $G$. Let $G_{\ell'}$ denote the subset of $G$ of elements with order coprime to $\ell$. Then $d^1: \Z \Irr (G) \ra \Z \IBr (G)$ denotes the decomposition map given by $d^1 (\chi) = \chi \circ \nu$ where $\nu$ is the characteristic function of $G_{\ell'}$.

\subsection{Fake $m$th Galois action}

Recall that a \textit{modular character triple} is a triple $(G, N, \theta)$ where $N \lhd G$ are finite groups and $\theta$ is a $G$-stable Brauer character of $N$. Suppose that $\mathcal{D}: G \rightarrow \tn{GL}_{\theta(1)}(\F)$ is a representation of $N$ affording $\theta$. Then there exists a projective representation $\mP$ of $G$ such that $\mP_N=\mathcal{D}$, and we can choose $\mP$ such that its associated factor set $\alpha: G \times G \rightarrow \F^{\times}$ satisfies $\alpha(g,n)=\alpha(n,g)=1$ for all $g \in G$ and $n \in N$. We say that such a projective representation $\mP$ is \textit{associated to the triple $(G, N, \theta)$}, or \textit {associated to $\theta$}.

\begin{definition}
	\label{def:squiggleapprox}
	Let $(G, N, \theta)$ and $( G, N, \theta')$ be modular character triples and let $m$ be a positive integer such that $(m, |N|) = 1$. Then we write
	\[ (G, N, \theta)^{(m)} \approx (G, N, \theta' )\]
	if there exist projective representations $\mP$ and $\mP'$ of $G$ associated to $\theta$ and $\theta'$ respectively, such that
	\begin{itemize}
		\item for every $c \in \mathrm{C}_{G}(N)$ the scalar matrices $\mP(c)$ and $\mP'(c)$ are associated to some $m'$ roots of unity $\xi$ and $\xi^m$;
		\item the factor sets $\alpha$ and $\alpha'$ of $\mP$ and $\mP'$ respectively take as values $m'$ roots of unity and satisfy
		\[ \alpha(g, g')^m = \alpha'(g, g')\]
		for every $g, g' \in G$.
	\end{itemize}
\end{definition}

\begin{definition}
	\label{def:fakegalaction}
	Let $N \lhd G$. There \textit{exists a fake $m$th Galois action} on $\IBr(N)$ with respect to $N \lhd G$ if there exists a $G$-equivariant bijection
	\[ f_m: \tn{ IBr}(N) \ra \tn{ IBr}(N)\]
	such that for every $\theta \in \IBr(N)$, 
	\[ ( G_{\theta}, N, \theta)^{(m)} \approx ( G_{\theta}, N, f_m(\theta) ) .\]
\end{definition}
 
 Let $\mP: G \rightarrow \tn{GL}_n(\F)$ be a projective representation with factor set $\alpha: G \times G \rightarrow \F^{\times}$. For an element $x \in G$ we let $\mP^x$ denote the projective representation given by $\mP^x(g) = \mP(g^{x^{-1}})$ for all $g \in G$.

\begin{lemma}\label{lem:conjsquiggle}
	Let $N \lhd G$ be finite groups. Let $\theta \in \IBr(N)$ and let $f: \IBr(N) \rightarrow \IBr(N)$ be a $G$-equivariant bijection. Then $(G_{\theta}, N, \theta )^{(m)} \approx (G_{\theta}, N, f(\theta) )$ if and only if $(G_{\theta^g}, N,\theta^g )^{(m)} \approx (G_{\theta^g}, N, f(\theta^g) )$ for any $g \in G$.
\end{lemma}

\begin{proof}
Let $\mP$ and $\mP'$ be the projective representations associated to $\theta$ and $f(\theta)$ respectively such that $(G_{\theta}, N, \theta )^{(m)} \approx (G_{\theta}, N, f(\theta) )$ as in Definition~\ref{def:squiggleapprox}. Then for any $g \in G$, $\mP^g$ and $\mP'^g$ are projective representations associated to $\theta^g$ and $f(\theta)^g$ respectively, and they satisfy the properties of Definition~\ref{def:squiggleapprox} so $(G_{\theta^g}, N,\theta^g )^{(m)} \approx (G_{\theta^g}, N, f(\theta^g) )$.
\end{proof}

% ------------------------------------------------
\subsection{Results from group theory}\label{subsec:grouptheory}
% ------------------------------------------------

Recall that every subgroup $U$ of $G \times H$ has the form	
\[U= \{(g,h) \in G_1 \times H_1 \mid \varphi(g G_2)=h H_2 \} ,\]
for some $5$-tuple $(G_1,G_2,H_1,H_2,\varphi)$ where $G_2 \lhd G_1 \leq G$ and $H_2 \lhd H_1 \leq H$ and $\varphi: G_1/G_2 \to H_1/H_2$ is an isomorphism. We note that $U$ is a subgroup of $G_1 \times H_1$, and $G_2 \times 1$ and $1 \times H_2$ are normal subgroups of $U$.

\begin{lemma}\label{subgroupdirectproduct}
	Let $G$ be a finite group and let $a \in \N_{>0}$. Suppose that $U$ is a subgroup of $G\times C_a$ associated to a 5-tuple $(G_1,G_2,H_1,H_2,\varphi)$ such that $G_1$ is abelian or $G_1 = G_2 \rtimes K$ for some finite group $K$. Then $U \cong G_2 \rtimes H_1$.
\end{lemma}

\begin{proof}
	If $G_1$ is abelian then $U$ is a subgroup of an abelian group so the result follows from the classification of finite abelian groups. Suppose now that $G_1 = G_2 \rtimes K$ for some finite group $K$. 
	
	Consider the projection map $\pi: U \to H_1$, given by $\pi(g,h) = h$. Since ker$(\pi) = G_2 \times 1$ we obtain an exact sequence,
		\[1 \to G_2 \times 1 \to U \to H_1 \to 1.\]
	Let $H_1 = \langle h_1 \rangle$. Consider the map $\pi': H_1 \to U, \, h_1 \mapsto (g_1,h_1)$ where $g_1 \in K$ is such that $\varphi(g_1 G_2) = h_1 H_2$. Since $K \cong G_1 / G_2 \cong H_1 / H_2$, the order of $h_1$ equals the order of $(g_1, h_1)$ and hence this map is well defined. As $\pi \circ \pi' = $ id, $\pi'$ is a section of $\pi$ and hence $U\cong G_2 \rtimes H_1$ as required. 
\end{proof}

\begin{corollary}\label{cor:structureofsylows}
	Let $a \in \N_{>0}$ and let $U \leq S_4 \times C_a$ be associated to a 5-tuple $(G_1,G_2,H_1,H_2,\varphi)$. Then $U \cong G_2 \rtimes H_1$. Moreover, if $A_4 \leq G_1$, then $V_4 \times 1 \lhd U$.
\end{corollary}

\begin{proof}
	By Lemma \ref{subgroupdirectproduct}, we need to show that for any $U \leq S_4 \times C_a$ associated to a 5-tuple $(G_1,G_2,H_1,H_2,\varphi)$, either $G_1$ is abelian or $G_1 = G_2 \rtimes K$ for some finite group $K$. Clearly if $G_2$ is trivial, then the second condition automatically holds. Since $H_1 \leq C_a$ is cyclic and $G_1 / G_2 \cong H_1 / H_2$, it follows that $G_1 / G_2$ is cyclic. It is therefore enough to consider each of the possibilities for $G_1$ and $G_2$ such that $G_1$ is not abelian, $G_2$ is not trivial, and $G_1/G_2$ is cyclic. 
	
	First let  $G_1 = S_4$. Then $G_2=A_4$ and $G_1=G_2 \rtimes \langle (1,2)\rangle$ is a semidirect product. If $G_1 = S_3$ then $G_2 \cong A_3$ and hence $G_1 = G_2 \rtimes \langle (1,2) \rangle$. Now suppose that $G_1 = D_8 = \langle a,b \mid a^4=b^2=1,\, [b,a]=a^2 \rangle$. Then $G_2$ is of order 2 or 4. If $|G_2|= 2$ then $G_2 = \langle a^2 \rangle$ and $G_1/G_2 \cong C_2 \times C_2$ is not cyclic. Therefore $G_2$ is of order 4 and so either $G_2 =\langle a \rangle$ and $G_1= G_2 \rtimes \langle b \rangle$; $G_2=\langle a^2,b \rangle$ and $G_1= G_2 \rtimes \langle ab \rangle$; or $G_2= \langle ab,a^2\rangle$ and $G_1=G_2 \rtimes \langle b \rangle$. 	Finally, suppose that $G_1 = A_4$. Then $G_2 = V_4$ and $G_1 = G_2 \rtimes \langle (1, 2, 3) \rangle$. We note that whenever $A_4 \leq G_1$ we have $V_4 \leq G_2$ and therefore $V_4 \times 1 \lhd U$.
\end{proof}

\subsection{Results from ordinary and projective character theory}

\begin{lemma}
	\label{lem:twocovers}
	Let $N \lhd G$ be finite groups. Let $\theta \in  \IBr(N)$ and let $\chi, \chi' \in \IBr(G \mid \theta)$. Then if $G/N$ is abelian, or if $\chi$ and $\chi'$ are extensions of $\theta$, there exists a linear character $\lambda \in \IBr(G/N)$ such that $\chi = \lambda\chi' $. 
\end{lemma}

\begin{proof} 
	By definition, $\chi$ and $\chi'$ are irreducible constituents of $\theta^G$. It then follows from \cite[Corollary 8.20]{N} that there exists an irreducible character $\lambda \in \IBr(G/N)$ such that $\chi =\lambda \chi' $. If $G/N$ is abelian, then $\lambda$ is a linear character. If $\chi$ and $\chi'$ are both extensions of $\theta$ then $\chi$ and $\chi'$ have the same degree. Therefore, $\lambda$ is again a linear character.
\end{proof}

If $\lambda:G_{\ell'} \to \mathbb{C}^\times$ is a linear Brauer character of $G$ we denote (by abuse of notation) $\lambda: G \to \mathbb{F}^\times$ the unique representation affording $\lambda$. The following is a generalisation of ideas present in the proof of \cite[Proposition 5.2]{N/S/T}. 

\begin{lemma}\label{lem:projrepgen}
	Let $G$ be a finite group. Suppose that $H_1 \lhd G$ and $G=H_1 H_2$. Let $H=H_1 \cap H_2$ and let $\theta \in \IBr(H)$ be a $G$-stable Brauer character of $H$. Suppose that there exists an extension $\theta_1 \in \IBr(H_1)$ of $\theta$. Let $\mD_1: H_1 \ra \tn{GL}_{\theta(1)}(\F)$ be a representation of $H_1$ affording $\theta_1$ and let $\mP_2: H_2 \ra \tn{GL}_{\theta(1)}(\F)$ be a projective representation associated to $\theta$ with factor set $\alpha_2$, such that $\mD_1$ and $\mP_2$ agree on $H$. Then there exists a well defined projective representation $\mP: G \ra \tn{GL}_{\theta(1)}(\F)$ associated to $\theta$ with factor set $\alpha$ such that for all $g = h_1h_2, g'= h_1'h_2' \in G$, $h_1, h_1' \in H_1$, $h_2, h_2' \in H_2$, 
	\[ \mP(g) = \mD_1(h_1) \mP_2(h_2) \]
	and $\alpha$ satisfies
	\[ \alpha(g ,g' ) = \lambda_{h_2}( h_1')\alpha_2(h_2, h_2')^{-1} \]
	where $\lambda_{h_2}$ is the unique linear Brauer character of $\IBr(H_1/H)$ such that $\theta_1^{h_2} \lam_{h_2} = \theta_1$. 
\end{lemma}

\begin{proof}
Let $x_1 \in H_1$ and $x_2 \in H_2$ such that $g=h_1 h_2 = x_1 x_2 \in G$. Then $h_1^{-1} x_1= h_2 x_2^{-1}:= h \in H$. Thus $\mP(h_1 h_2)=\mP(h_1 h h^{-1} h_2)=\mP(x_1 x_2)$, so $\mP$ is well defined.
	
	Now let $g = h_1h_2$ and $g'= h_1'h_2' \in G$. Then
	\begin{align*}
	\mP(g)\mP(g') & = \mD_1(h_1)\mP_2(h_2)\mD_1(h_1')\mP_2(h_2') \\
	& = \mD_1(h_1)\mD_1(h_1')^{\mP_2(h_2)^{-1}}\mP_2(h_2)\mP_2(h_2').
	\end{align*}
	We define the following two maps for this fixed $h_2$. 
	\[ \begin{array}{rrclccrrcl}
	\nu_1: & H_1 & \ra & GL_{\theta(1)}(\F) &&& 
	\nu_2: & H_1 & \ra & GL_{\theta(1)}(\F) \\
	& h & \mapsto & \mD_1(h)^{\mP_2(h_2)^{-1}} &&&
	& h & \mapsto & \mD_1(h_2 h h_2^{-1})
	\end{array}
	\]
	The maps $\nu_1$ and $\nu_2$ are representations of $H_1$ associated to $\theta_1$ and $\theta_1^{h_2}$ respectively. Since $\theta$ is $G$-stable, $\theta^{h_2} = \theta$ so $\theta_1$ and $\theta_1^{h_2}$ extend the same character $\theta$ of $H$. Hence there exists a linear character $\lam_{h_2} \in \IBr(H_1 / H)$ such that $\theta_1^{h_2}\lam_{h_2} =  \theta_1$, by Lemma~\ref{lem:twocovers}.
	
Recall our convention that we also denote the representation affording $\lam_{h_2}$ by $\lam_{h_2}$. The representations $\nu_2 \lam_{h_2}$ and $\nu_1$ both afford $\theta_1$, so they are similar representations of $H_1$. Since $\nu_2 \lam_{h_2}$ and $\nu_1$ coincide on $H$, it follows that $\nu_2 \lam_{h_2}= \nu_1$. We therefore have the following:
	
	\begin{align*}
	\mP(g)\mP(g') 
	& = \mD_1(h_1)\nu_1(h_1')\mP_2(h_2)\mP_2(h_2') \\
	& = \mD_1(h_1) (\nu_2\lam_{h_2})(h_1') \mP_2(h_2) \mP_2(h_2')\\
	& = \lam_{h_2}(h_1')\alpha_2(h_2, h_2')^{-1} \mP(gg').			 
	\end{align*}
	Thus $\mP$ is a projective representation with factor set $\alpha$ satisfying $\alpha(g, g') =  \lam_{h_2}(h_1')\alpha_2(h_2, h_2')^{-1}$ for all $g = h_1 h_2$, $g' = h_1' h_2' \in G$ where $\lam_{h_2} \in \IBr (H_1/H)$ is the linear character determined by $\theta_1^{h_2} \lam_{h_2} = \theta_1$ and $\alpha_2$ is the factor set of $\mP_2$. Since $\mP|_{H}$ is an ordinary representation of $G$ affording $\theta$, $\mP$ is a projective representation of $G$ associated to $\theta$.	
\end{proof}

% ----------------------------------------------------------------
% ----------------------------------------------------------------
\subsection{Finite groups of Lie type}\label{sec:fglt}
% ----------------------------------------------------------------
% ----------------------------------------------------------------

Let $p \neq \ell$ be a prime and let $\bG$ be a simple simply connected algebraic group defined over $\overline \F_p$. Let $\Phi$ be a root system of $\bG$ relative to a maximal torus $\bT$ with base $\Delta$. For $\alpha \in \Phi$ let $x_{\alpha}$ be the associated one-parameter subgroup. Let $F_0: \bG \to \bG$ be the field automorphism defined by $x_\alpha(t) \mapsto x_\alpha(t^p)$ for all $\alpha \in \Phi$ and $t \in \F$. For a symmetry $\tau$ of the Dynkin diagram associated to $\Delta$ we let $\tau:\bG \to \bG$ be the graph automorphism given by $x_\alpha(t) \mapsto x_{\tau (\alpha)}(t)$. Let $F: \bG \ra \bG$ be a Frobenius endomorphism of the form $F=F_0^f \tau$, where $\tau$ is a graph automorphism as above and $q = p^f$, endowing $\bG$ with an $\F_q$-structure. Let $\bG^F$ be the finite group of the fixed points of $\bG$ under $F$. We say that $\bG^F$ is \textit{untwisted} if $F=F_0^f$ and \textit{twisted} otherwise.

Let $\bG^*$ be a connected reductive algebraic group with maximal torus $\bT^*$ such that $(\bG^*,\bT^*, F^* )$ is dual to $( \bG,\bT, F )$ in the sense of \cite[Definition 13.10]{D/M} and the subsequent discussion. Since it will be clear which Frobenius we are referring to, we drop the * notation and just write $F$ for both Frobenius maps.

Let $\iota: \bG \hookrightarrow \wtG$ be a regular embedding of algebraic groups, with $\wtG$ a connected reductive algebraic group such that $\tn{Z}(\wtG)$ is connected and $\left[ \wtG, \wtG \right] \subseteq \bG$ (see \cite[Section 15.1]{C/E3}). The Frobenius endomorphism $F$ extends to $\tilde \bG$ and we also denote by $F: \wtG \to \wtG$ a fixed extension of $F$ to $\wtG$. Then $\wtT = \tn{Z}(\wtG).\bT$ is an $F$-stable maximal torus of $\wtG$ and we let $\wtG^*$ and $\wtT^*$ be such that $(\wtG^*, \wtT^*, F)$ is dual to $(\wtG, \wtT, F)$. Let $\iota^{\ast}: \wtG^* \ra \bG^*$ be a surjective morphism dual to $\iota$.

To each $\bG^{*F}$-conjugacy class $(s)$ of a semisimple element $s \in \bG^{*F}$ we have an associated rational Lusztig series $\mathcal{E}(\bG^F,(s)) \subseteq \Irr(\bG^F)$, see remarks before \cite[Proposition 14.41]{D/M}. In addition we let $\mathcal{E}(\bG^F, \ell')$ be the union of rational Lusztig series $\mathcal{E}(\bG^F,(s))$ where $s$ runs over all semisimple elements of $\bG^{*F}$ with order coprime to $\ell$.

% ----------------------------------------------------------------
% ----------------------------------------------------------------
\section{Finite groups of Lie type of type A}\label{sec:TypeA}
% ----------------------------------------------------------------
% ----------------------------------------------------------------

For this section we let $\bG = \tn{SL}_n(\overline{\F}_p)$,  $\wtG = \tn{GL}_n(\overline{\F}_p)$ with $n > 1$. Let $m \in \N$ be such that $(m, |\bG^F|)= 1$ and therefore $(m, |\wtG^F|) = 1$. 
Let $D = \langle F_0,\tau \rangle$ be the group of automorphisms of $\wtGF$ generated by field and graph automorphisms. 

\begin{remark}
	\label{rmk:stabilizers}
	Let $\Theta: \IBr(\bG^F) \ra \Irr(\bG^F)$ denote the injective map given in \cite[Definition 2.2]{Den}. It follows from \cite[Theorem A and Lemma 2.3]{Den} that $\Theta(\IBr(\bG^F))$ is an Aut$(\bG^F)$-stable unitriangular basic set for $\bG^F$ and $\Theta$ is $\Aut(\bG^F)$-equivariant. Hence for all $H \leq $ Aut$(\bG^F)$ and all $\theta \in \IBr(\bG^F)$, $H_{\theta} = H_{\Theta(\theta)}$.
\end{remark}

\begin{lemma}
	\label{lem:starcondition}
	For any $\chi \in \IBr(\wtGF)$ there exists a $\theta \in \IBr(\bG^F \mid \chi)$ such that 
\begin{enumerate}
		\item[$(i)$] $(\wtGF \rtimes D)_{\theta} = \wtGF_{\theta} \rtimes D_{\theta}$, and 
		\item [$(ii)$] $\theta$ extends to $(\bG^F \rtimes D)_{\theta}$.
	\end{enumerate}
\end{lemma}

\begin{proof}
	Let $\chi \in \IBr(\wtGF)$, $\theta \in \IBr(\bG^F \mid \chi)$ and $\psi = \Theta(\theta)$. By \cite[Theorem 4.1]{C/S}, there exists a $g \in \wtGF$ such that $\psi^g$ satisfies (i) and (ii). We claim that $\theta^g \in \IBr(\bG^F \mid \chi)$ satisfies (i) and (ii). Since $\Theta(\theta^g) = \Theta(\theta)^g = \psi^g$, it follows from Remark~\ref{rmk:stabilizers} that $(\wtGF \rtimes D)_{\theta^g} = (\wtGF \rtimes D)_{\psi^g} = \wtGF_{\psi^g} \rtimes D_{\psi^g} = \wtGF_{\theta^g} \rtimes D_{\theta^g}$ so $\theta^g$ satisfies (i). 
	Since $\psi^g$ satisfies part (ii), there exists an extension $\tilde{\psi} \in \Irr(\bG^F \rtimes D)_{\psi^g}$ of $\psi^g$. By Remark~\ref{rmk:stabilizers} and (i) we have that $(\bG^F \rtimes D)_{\psi^g} = \bG^F \rtimes D_{\theta^g}$. Then since $d^1(\tilde{\psi})_{\bG^F} = d^1(\tilde{\psi}_{\bG^F} ) = d^1(\psi^g)$, $\theta^g$ is a constituent of $d^1(\tilde{\psi})_{\bG^F}$ of multiplicity 1. It follows that there exists some irreducible constituent $\eta \in \IBr(\bG^F \rtimes D_{\theta^g})$ of $d^1 (\tilde{\psi})$ such that $\theta^g$ is an irreducible constituent of $\eta_{\bG^F}$ of multiplicity 1. Since $\theta^g$ is stable in $\bG^F \rtimes D_{\theta^g}$, it follows that $\eta$ is an extension of $\theta^g$ to $\bG^F \rtimes D_{\theta^g}$, so $\theta^g$ satisfies (ii).
\end{proof}

\begin{lemma}
	\label{lem:ftilde}
	There exists a $D$-equivariant bijection
	\[ \tilde{f}_m: \tn{ IBr} (\wtGF) \ra \tn{ IBr} (\wtGF) \]
	such that $\tilde{f}_m(\lam \chi) = \lam^m \tilde{f}_m(\chi)$ for all $\chi \in \IBr(\wtGF)$ and all $\lam \in \IBr(\wtGF \mid 1_{\bG^F})$, and $\tilde{f}_m (\IBr(\wtGF\mid\nu)) = \IBr(\wtGF \mid \nu^m)$ for all $\nu \in \IBr(Z(\wtGF))$.   
\end{lemma}

\begin{proof}
	
Recall that there exists a bijection $\tn{Z}(\tilde \bG^*)^F \rightarrow \Irr(\tilde \bG^F \mid 1_{\bG^F})$, $z \mapsto \hat z$ , see for example \cite[Proposition 13.30]{D/M}. The set $\mathcal{E}(\wtGF,\ell')$ is a basic set of $ \Irr(\wtGF)$ and we can define an injective map $\tilde{\Theta}:\IBr(\wtGF) \to  \Irr(\wtGF)$ with image $\mathcal{E}(\wtGF,\ell')$ such that the decomposition matrix associated to this map is unitriangular \cite[2.4]{G3}. For every $z \in Z(\wtG^*)_{\ell'}^F$ we have that $\tilde \Theta (\psi d^1(\hat z)) = \tilde \Theta(\psi) \hat z$ for all $\psi \in \IBr(\wtG^F)$. Let $\sigma \in $ Gal$(\Q_{|\wtG^F|}/\Q)$ be a Galois automorphism such that $\sigma(\zeta) = \zeta^m$, where $\zeta$ is a primitive $|\wtG^F|$-th root of unity. Define a map
\begin{align*}
\tilde{f}_m:  \IBr(\wtGF) & \ra \IBr(\wtGF) \\
			 \chi & \mapsto  \tilde{\Theta}^{-1}(\tilde{\Theta}(\chi)^\sigma).
\end{align*}
Since $\mathcal{E}(\wtGF, \ell')^\sigma = \mathcal{E}(\wtGF, \ell')$ by \cite[Lemma 9.1]{Na/Ti}, this map is well defined. 	

Let $\lam \in \IBr(\wtGF \mid 1_{\bG^F})$ and let $z \in \mathrm{Z}(\wtG^\ast)_{\ell'}^F$ be the central element such that $d^1(\hat z) = \lambda$. Then for all $\chi \in \IBr(\wtGF)$, 
\[\tilde f_m (\lambda \chi) =  \tilde{\Theta}^{-1}(\tilde{\Theta}(\chi \lambda)^\sigma) = \tilde{\Theta}^{-1}((\hat z \tilde{\Theta}(\chi) )^{\sigma}) 
= d^1(\hat z)^m \tilde{\Theta}^{-1}(\tilde{\Theta}(\chi) ^{\sigma}) = \lambda^m \tilde f_m(\chi).\]
Now let $\nu \in \Irr(\mathrm{Z}(\wtGF))$. Then since $\tilde{\Theta}(\IBr(\wtGF) \mid d^1(\nu)) \subseteq \IBr(\wtGF \mid \nu)$ and $\Irr(\wtGF \mid \nu) = $ $\Irr(\wtGF \mid \nu^m)$, it follows that $\tilde{f}_m (\IBr(\wtGF\mid\nu)) =$  $\IBr(\wtGF \mid \nu^m)$.
\end{proof}

\begin{lemma}
	\label{lem:Dtheta}
	Let $\theta \in \IBr(\bG^F)$ and let $\chi \in  \IBr (\wtGF \mid \theta)$. Suppose that $(\wtGF \rtimes D)_{\theta} = \wtGF_{\theta} \rtimes D_{\theta}$. Then 
	\[ D_{\theta} = \{ d \in D \mid \chi^d = \chi \lam \tn{ \tn{for some} } \lam \in \tn{ IBr}(\wtGF / \bG^F)\}. \]
\end{lemma}

\begin{proof}
	If $d \in D_{\theta}$ then $\chi$ and $\chi^d $ are both elements of $\IBr(\wtGF \mid \theta)$ so $\chi^d = \chi\lam$ for some linear $\lam \in \IBr(\wtGF / \bG^F)$ by Lemma~\ref{lem:twocovers}. Hence $D_{\theta} \subseteq \{ d \in D \mid \chi^d = \chi \lam \tn{ for some } \lam \in \tn{ IBr}(\wtGF / \bG^F)\}$. 
	Now suppose that $d \in D$ is such that $\chi^d = \chi\lam$ for some $\lam \in \IBr(\wtGF / \bG^F)$. Then $\theta^d$ and $\theta$ are both elements of $\IBr(\bG^F \mid \chi^d)$, so $\theta^d = \theta^g$ for some $g \in \wtGF$. Therefore $g^{-1}d \in (\wtGF \rtimes D)_{\theta}$. Since $(\wtGF \rtimes D)_{\theta} = \wtGF_{\theta} \rtimes D_{\theta}$, it follows that $g \in \wtGF_{\theta}$ and $d \in D_{\theta}$.
\end{proof}

Let $\mathcal{G}$ be a set of representatives of the $(\wtGF \rtimes D)$-orbits of characters in $\IBr(\bG^F)$ satisfying Lemma~\ref{lem:starcondition} (i) and (ii). For $\theta \in \IBr(\bG^F)$, let $\theta_0$ denote the unique character in $\mG$ such that $\theta = \theta_0^{gd}$ for some $gd \in \wtGF \rtimes D$, and let $\tilde \theta_0 \in \IBr(\wtGF \mid \theta_0)$. 
There exists a unique character $\chi_0' \in \mathcal{G}$ and element $d' \in D$ such that ${\chi_0'}^{d'} \in \IBr(\bG^F \mid \tilde{f}_m(\tilde \theta_0))$.
Since $(\wtGF \rtimes D)_{{\chi_0'}^{d'}} = d'(\wtGF_{\chi_0'}) \rtimes d'(D_{\chi_0'})$, the character ${\chi_0'}^{d'}$ also satisfies the conditions of Lemma~\ref{lem:starcondition} (i) and (ii).
	We let $\chi_0:= {\chi_0'}^{d'}$ and define a map $f_m : \IBr(\bG^F) \ra \IBr(\bG^F)$ by 
	\[f_m(\theta) := \chi_0^{gd}.\]

\noindent
We claim that this is a well defined $(\wtG^F \rtimes D)$-equivariant bijection. First suppose that $\tilde \theta_0$ and $\tilde \theta_0'$ are two elements of $\IBr(\wtGF \mid \theta_0)$. Then since $\wtGF / \bG^F$ is abelian, $\tilde \theta_0 = \lam \tilde \theta_0'$ for some linear character $\lam \in \IBr(\wtGF / \bG^F)$ by Lemma~\ref{lem:twocovers}. Further, $\tilde{f}_m(\tilde \theta_0) = \lam^m \tilde{f}_m(\tilde \theta_0')$ so $\IBr(\bG^F \mid \tilde{f}_m (\tilde \theta_0)) = \IBr (\bG^F \mid \tilde{f}_m (\tilde \theta_0'))$. Hence $\chi_0$, and therefore $f_m(\theta)$, is independent of the choice of $\tilde \theta_0 \in \IBr(\wtGF \mid \theta_0)$. Thus the map is well defined, and $(\wtG^F \rtimes D)$-equivariant by construction.

\begin{lemma}\label{stabilizer}
In the notation as above we have	$(\wtGF \rtimes D)_{\theta_0} = (\wtGF \rtimes D)_{\chi_0}$. 
\end{lemma}

\begin{proof}
	
	Since $\theta_0$ and $\chi_0$ satisfy Lemma~\ref{lem:starcondition} (i) and (ii), $(\wtGF \rtimes D)_{\theta_0} = \wtGF_{\theta_0} \rtimes D_{\theta_0}$ and $(\wtGF \rtimes D)_{\chi_0} = \wtGF_{\chi_0} \rtimes D_{\chi_0}.$ By Lemma~\ref{lem:Dtheta}, 
	\[ D_{\theta_0} = \left\{ d \in D \mid \tilde \theta_0^d = \tilde \theta_0 \lam \tn{ for some } \lam \in \tn{ IBr}(\wtGF/\bG^F) \right\}.\] 
	Therefore, since $\tilde{f}_m$ is a $D$-equivariant bijection and $\tilde{f}_m(\tilde \theta_0) \in \IBr(\wtGF \mid \chi_0)$,
	\begin{align*}
	D_{\theta_0} = & ~ \left\{ d \in D \mid \tilde{f}_m(\tilde \theta_0^d) = \tilde{f}_m(\tilde \theta_0\lam) \tn{ for some } \lam \in \tn{ IBr}(\wtGF/\bG^F) \right\}\\
	= & ~ \left\{ d \in D \mid \tilde{f}_m(\tilde \theta_0)^d = \tilde{f}_m(\tilde \theta_0)\lam^m \tn{ for some } \lam \in \tn{ IBr}(\wtGF/\bG^F) \right\}\\
	= & ~ D_{\chi_0}.
	\end{align*}
	Recall that $\wtGF_{\theta_0} = \bigcap \ker(\lam)$ where $\lam$ runs over the elements of Stab$_{\IBr(\wtG^F/\bG^F)}(\tilde \theta_0)$.  Since $\tilde f_m$ is a bijection and $\tilde{f}_m(\lam \tilde \theta_0) = \lam^m \tilde{f}_m(\tilde \theta_0)$, $\lam$ stabilises $\tilde \theta_0$ if and only if $\tilde f_m(\lam) = \lam^m$ stabilises $\tilde f_m(\tilde \theta_0)$. Thus, since $(m, |\wtGF|) = 1$,  $\wtGF_{\theta_0} = \bigcap \ker(\lam^m)$ where $\lam^m$ now runs over the elements of  Stab$_{\IBr(\wtG^F/\bG^F)}(\tilde f_m( \tilde \theta_0))$. In particular, $\wtGF_{\theta_0} = \wtGF_{f_m(\theta_0)}$ and the result follows. 
\end{proof}

Since $\tilde f_m$ is a bijection such that $\tilde{f}_m(\lam \chi) = \lam^m \tilde{f}_m(\chi)$ for all $\chi \in \IBr(\wtGF)$ and all $\lam \in \IBr(\wtGF \mid 1_{\bG^F})$, by construction the map $f_m|_{\mG}: \mG \ra \IBr(\bG^F)$ is an injection. Hence by Lemma~\ref{stabilizer}, $f_m: \IBr(\bG^F) \ra \IBr(\bG^F)$ is a $(\wtGF \rtimes D)$-equivariant bijection, proving the claim. It follows that $(\wtGF \rtimes D)_{\theta} = (\wtGF \rtimes D)_{f_m(\theta)}$ for all $\theta \in \IBr(\bG^F)$.

\begin{proposition}\label{prop:typeA}
The map $f_m : \tn{ IBr}(\bG^F) \ra \tn{ IBr}(\bG^F)$ defines a fake $m$th Galois action on $\IBr(\bG^F)$ with respect to $\bG^F \lhd \wtGF \rtimes D$. 
\end{proposition}

\begin{proof}
	
	Since $f_m : \tn{ IBr}(\bG^F) \ra \tn{ IBr}(\bG^F)$ is a $(\wtGF \rtimes D)$-equivariant bijection, by Definition~\ref{def:fakegalaction} and Lemma~\ref{lem:conjsquiggle} it remains to show that 
	\[ ( (\wtGF \rtimes D)_{\theta}, \bG^F, \theta )^{(m)} \approx   ( (\wtGF \rtimes D)_{\theta}, \bG^F, f_m(\theta) ) \]
for every $\theta \in \mathcal{G}$.
	
	Let $\theta \in \mathcal{G}$. Since $\wtG^F/\bG^F$ is cyclic there exists an extension $\theta_1$ of $\theta$ to $\wtG^F_{\theta}$. Since $\theta$ satisfies Lemma~\ref{lem:starcondition} (i) and (ii), $(\wtG^F \rtimes D)_{\theta} = \wtG^F_{\theta} \rtimes D_{\theta}$ and there exists an extension $\theta_2$ of $\theta$ to $\bG^F \rtimes D_{\theta}$. Let $H_1 = \wtG^F_{\theta}$ and let $H_2 = \bG^F \rtimes D_{\theta}$. Let $\mD_1$ be an ordinary representation of $\wtG_\theta^F$ affording $\theta_1$. Let $\mD_2$ be an ordinary representation of $\bG^F \rtimes D_{\theta}$ affording $\theta_2$ such that $\mD_1$ and $\mD_2$ agree on $\bG^F$. Then by Lemma~\ref{lem:projrepgen} there exists a projective representation $\mP: (\wtG^F \rtimes D)_{\theta} \rightarrow GL_{\theta(1)}(\F)$ associated to $\theta$ with factor set $\alpha$ such that  
	\[  \hspace{5ex}\mP(x) = \mD_1(h_1) \mD_2(h_2), \hspace{3ex} \mbox{and} \]
	\[ \alpha(x ,x' ) = \lambda_{h_2}( h_1'), \]
	for all $x = h_1h_2, x'= h_1'h_2' \in (\wtG^F \rtimes D)_{\theta}$, $h_1, h_1' \in \wtG^F_{\theta}$, $h_2, h_2' \in \bG^F \rtimes D_{\theta}$, where $\lambda_{h_2}$ is the linear Brauer character of $\wtG^F_{\theta}/\bG^F$ such that $\theta_1^{h_2} \lam_{h_2} = \theta_1$. 
	
	Let $\tilde \theta_1 \in \IBr(\wtGF)$ correspond to $\theta_1\in \IBr(\wtGF_{\theta} \mid \theta)$ via Clifford correspondence with respect to $\theta$ \cite[Theorem 8.9]{N}. Let $\theta_1' \in \IBr(\wtG^F_{f_m(\theta)} \mid f_m(\theta))$ denote the Clifford correspondent of $\tilde f_m(\tilde \theta_1)$ with respect to $f_m(\theta)$. Since $f_m(\theta)$ satisfies Lemma~\ref{lem:starcondition} (i) and (ii), there exists an extension $\theta_2'$ of $f_m(\theta)$ to $\bG^F \rtimes D_{\theta}$. Let $\mD'_1$ be an ordinary representation of $\wtG_{\theta}^F$ associated to $\theta_1'$. Let $\mD'_2$ be an ordinary representation of $\bG^F \rtimes D_{\theta}$ affording $\theta_2'$ such that $\mD'_1$ and $\mD'_2$ agree on $\bG^F$. Then by Lemma~\ref{lem:projrepgen} again, there exists a projective representation $\mP': (\wtG^F \rtimes D)_{\theta} \rightarrow GL_{\theta(1)}(\F)$ associated to $f_m(\theta)$ with factor set $\alpha'$ such that  
	\[ \hspace{5ex}  \mP'(x) = \mD'_1(h_1) \mD'_2(h_2) ,  \hspace{3ex} \mbox{and}\]
	\[ \alpha'(x ,x' ) = \lambda'_{h_2}( h_1'), \]
	for all $x = h_1h_2, x'= h_1'h_2' \in (\wtG^F \rtimes D)_{\theta}$, $h_1, h_1' \in \wtG^F_{\theta}$, $h_2, h_2' \in \bG^F \rtimes D_{\theta}$, where $\lambda'_{h_2}$ is the linear Brauer character of $\IBr(\wtG^F_{\theta}/\bG^F)$ such that $\theta_1'^{h_2} \lam'_{h_2} = \theta_1'$. 
	
	Let $\tilde \lambda_{h_2}' \in \IBr(\wtGF/\bG^F)$ be an extension of $\lambda_{h_2}'$ to $\wtGF$. It follows from Lemma~\ref{lem:ftilde} that $\tilde{f}_m(\tilde \theta_1)^{h_2} \tilde \lam'_{h_2} = \tilde{f}_m(\tilde \theta_1)$. Since $\tilde{f}_m$ is $D$-equivariant, we have the following for any $h_2 \in \bG^F \rtimes D_{\theta}$,

	\[ \tilde{f}_m(\tilde \theta_1) \tilde\lambda_{h_2}'^{-1} = \tilde{f}_m(\tilde \theta_1)^{h_2} =
	 \tilde{f}_m(\tilde \theta_1 \tilde \lam_{h_2}^{-1}) = \tilde{f}_m(\tilde \theta_1)\tilde\lam_{h_2}^{-m} . \]	 
	 Hence $ \lambda_{h_2}' = \lam_{h_2}^{m} $ for each $h_2 \in \bG^F \rtimes D_{\theta}$, so $\alpha(x, x')^m = \alpha'(x, x')$ for all $x, x' \in (\wtG^F \rtimes D)_{\theta}$. 
		
	From the description of the action of automorphisms of $\bG^F$ given in \cite[Theorem 2.5.1]{G/L/SIII}, we observe that  $C_{\wtGF \rtimes D}(\bG^F) = Z(\wtG^F)$. Therefore $\mP|_{Z(\wtGF)} = \mD_1|_{Z(\wtGF)} = \nu Id$ where $\nu \in \IBr(Z(\wtGF) \mid \tilde \theta_1)$. Similarly, $\mP'|_{Z(\wtGF)} = \mD_1'|_{Z(\wtGF)} = \eta Id$ where $\eta \in \IBr(Z(\wtGF) \mid \tilde f_m(\tilde \theta_1))$. Since $\tilde{f}_m (\IBr(\wtGF\mid\nu)) = \IBr(\wtGF \mid \nu^m)$ for all $\nu \in \IBr(Z(\wtGF))$ by Lemma~\ref{lem:ftilde}, it follows that $\eta = \nu^m$. Therefore $((\wtGF \rtimes D)_{\theta}, \bG^F, \theta)^{(m)} \approx ((\wtGF \rtimes D)_{\theta}, \bG^F, f_m(\theta))$.
\end{proof}

\begin{corollary}\label{coro:typeA}
There exists a fake $m$th Galois action on $\IBr(\bG^F)$ with respect to $\bG^F \lhd \bG^F  \rtimes \mathrm{Aut}(\bG^F)$.
\end{corollary}

\begin{proof}
Let $H:=\wtG^F \rtimes D$. It follows from \cite[Theorem 2.5.1]{G/L/SIII} that there exists a surjective map \break $H \twoheadrightarrow \Aut(\bG^F)$ such that $H/C_{H}(\bG^F) \cong \mathrm{Out}(\bG^F)$. Therefore we can apply \cite[Corollary 4.12]{S/V} and conclude that there exists a fake $m$th Galois action on $\IBr(\bG^F)$ with respect to $\bG^F \lhd \bG^F \rtimes \Aut(\bG^F)$. 
\end{proof}

% ----------------------------------------------------------------
% ----------------------------------------------------------------
\section{Finite groups of Lie type of type B, C, D, E$_6$ and E$_7$}\label{sec:notofTypeA}
% ----------------------------------------------------------------
% ----------------------------------------------------------------

We continue with the notation of Section~\ref{sec:fglt} with $\bG$ a simple algebraic group of simply connected type and $F = F_0^f \tau$ a Frobenius endomorphism defining an $\mathbb{F}_q$-structure on $\bG$ with $q = p^f$. For this section we assume that $\bG$ is of type B, C, D, E$_6$, or E$_7$. 
In addition we assume that $\tau^2 = 1$, that is, we exclude the case where $\bG^F$ is of type ${}^3 $D$_4$. 

Let $\mathcal{L}: \bG \ra \bG$ denote the Lang map defined by $\mathcal{L}(g) = g^{-1} F(g)$ for all $g \in \bG$. 
Let Diag$({\bG^F})$ denote the group of diagonal automorphisms of $\bG^F$ in $\Out(\bG^F)$. By \cite[12.5]{Ta}, the proof of \cite[Proposition 1.5]{Le} together with \cite[Lemma 1.3]{D/L/M},
\[ \tn{Diag}({\bG^F}) \cong \wtG^F / (\tn{Z}(\wtG)^F \bG^F)  \cong \mathcal{L}^{-1}(\tn{Z}(\bG))/ (\tn{Z}(\bG ) \bG^F).\]
Hence each diagonal automorphism of $\bG^F$ can be realised as conjugation by some element of $\mathcal{L}^{-1}(\tn{Z}(\bG ))$. Since the Lang map is surjective on $\bG$ and the restriction $\mathcal{L}|_{\mathcal{L}^{-1}(\tn{Z}(\bG ))}$ is a group homomorphism with kernel $\bG^F$, we have $\mathcal{L}^{-1}(\tn{Z}( \bG ) ) / \bG^F \cong 
\tn{Z}(\bG )$. Let $\Gamma$ be the group of graph automorphisms which commute with $F$, regarded as automorphisms of $\mathcal{L}^{-1}(\tn{Z}(\bG ))$. 

Define $A := \mathcal{L}^{-1}(\tn{Z}(\bG )) \rtimes \langle F_0, \Gamma \rangle$. Then 
\begin{equation*}\label{eqn:AoverGF}
A/\bG^F =  \mathcal{L}^{-1}(\tn{Z}( \bG ) ) /\bG^F \rtimes \langle F_0, \Gamma \rangle \cong \tn{Z}( \bG ) \rtimes \langle F_0, \Gamma \rangle,
\end{equation*}
and for convenience we identify $A/\bG^F$ with $\tn{Z}(\bG) \rtimes \langle F_0, \Gamma \rangle$. By construction of $A$ there is a surjection 
\[ \delta : A/\bG^F \twoheadrightarrow \Out(\bG^F).\]
Let $z \in \tn{Z}(\bG)$ and let $x \in \mathcal{L}^{-1}(\tn{Z}(\bG ))$ be such that $\mathcal{L}(x) = z$. Then we let $\delta_z \in \Out(\bG^F)$ denote the diagonal automorphism of $\bG^F$ given by conjugation by $x$. In order to express $A/\bG^F$ in a more convenient form, we use the following lemma to define one more piece of notation.

\begin{lemma}\label{lem:actiononZG}
	If $F_0$ does not act trivially on $\tn{Z}(\bG)$ then there exists a graph automorphism $\gamma \in \Gamma$ such that $\gamma F_0$ acts trivially on $\tn{Z}(\bG)$. 
\end{lemma}

\begin{proof}
	 The result is immediate if $\tn{Z}(\bG)$ is trivial or of order 2 so there are just three cases to consider. We use the notation of \cite[Section 1.12]{G/L/SIII}. 
	 
	 First suppose that $\bG$ is of type E$_6$ and $p \neq 3$. Let $1 \neq \omega \in \overline{\F}_p^\times$ with $\omega^3=1$. By \cite[Table 1.12.6]{G/L/SIII}, $h=h_{\alpha_1}(\omega) h_{\alpha_2}(\omega^2) h_{\alpha_5}(\omega) h_{\alpha_6}(\omega^2)$  is a generator of $\tn{Z} ( \bG )$. If $p \equiv 1 \, \mathrm{mod}(3)$ then $\omega^p=\omega$ so $F_0$ acts trivially on $\tn{Z}(\bG)$. If $p \equiv 2 \, \mathrm{mod}(3)$ then $\omega^p=\omega^2$, so $F_0(h)=h^2$. Further, the non-trivial graph automorphism $\gamma$ acts on $\tn{Z}(\bG)$ by $\gamma(h)=h_{\alpha_6}(\omega) h_{\alpha_5}(\omega^2) h_{\alpha_2}(\omega) h_{\alpha_1}(\omega^2)=h^2$. Therefore $\gamma$ and $F_0$ have the same action of order 2 on $\tn{Z}(\bG)$, and hence $\gamma F_0$ acts trivially on $\tn{Z}(\bG)$.

Now suppose that $\bG$ is of type D$_{2m}$ for some $m \in \N$, and $p \neq 2$. Then $\tn{Z}(\bG)$ has two generators, $h_1 = h_{\alpha_1}(-1) h_{\alpha_3}(-1) \dots h_{\alpha_{2m-1}}(-1)$ and $h_2 = h_{\alpha_{2m -1}}(-1)h_{\alpha_{2m}}(-1)$. Since $p \neq 2$, the field automorphism $F_0$ stabilises $-1$, hence $F_0$ acts trivially on the centre. 

Finally, suppose that $\bG$ is of type D$_{2m+1}$ for some $m \in \N$, and $p \neq 2$. Let $1 \neq \omega \in \overline \F_p^\times$ such that $\omega^2 = -1$. Then $\tn{Z}(\bG)$ has one generator $h = h_{\alpha_1}(-1) h_{\alpha_3}(-1) \dots h_{\alpha_{2m-1}}(-1)h_{\alpha_{2m}}(\omega)h_{\alpha_{2m+1}}(-\omega)$. Again, since $p$ is odd, $F_0$ stabilises $-1$. If $p \equiv 1 $ mod $(4)$ then $F_0(\omega) = \omega$ so $F_0$ acts trivially on $\tn{Z}(\bG)$. If $p \equiv 3 $ mod $(4)$ then $F_0(\omega) = \omega^3 = -\omega$. The non-trivial graph automorphism $\gamma$ swaps $\alpha_{2m}$ and $\alpha_{2m+1}$ so $\gamma(h) = h_{\alpha_1}(-1) h_{\alpha_3}(-1) \dots h_{\alpha_{2m-1}}(-1)h_{\alpha_{2m+1}}(\omega)h_{\alpha_{2m}}(-\omega)$. In particular, $\gamma$ and $F_0$ have the same action on $\tn{Z}(\bG)$ and hence  $\gamma F_0$ acts trivially on $\tn{Z}(\bG)$.
\end{proof}

\begin{notation}
	If $F_0$ acts trivially on $\tn{Z}(\bG)$ then we let $F_0' := F_0$. Otherwise we let $F_0' := \gamma F_0$ where $\gamma$ is the non-trivial graph automorphism such that $\gamma F_0$ acts trivially on $\tn{Z}(\bG)$ as given in Lemma~\ref{lem:actiononZG}. 
\end{notation}
\noindent 
Then we have
\[ A/\bG^F \cong \tn{Z}( \bG ) \rtimes \langle F_0, \Gamma \rangle  = (\tn{Z}(\bG) \rtimes  \Gamma ) \times \langle F_0' \rangle,\]
where $\langle F_0' \rangle \cong C_a$ for some positive integer $a$.  As $\tn{Z}(\bG)$ is the $p'$-part of $\Lambda(\Phi)$ given in \cite[Table 9.2]{M/T}, and $\Gamma$ can be determined by the Dynkin diagram of $\bG$, we observe the following. 

\begin{corollary}\label{cor:AoverGF} ~
There exists a positive integer $a$ such that $A/ \bG^F$ has the following form: 
	\begin{itemize}
		\item If $\bG^F$ is of type \tn{D}$_4$ then $A/\bG^F$ is a subgroup of $(V_4 \rtimes S_3) \times C_a \cong S_4 \times C_a$.
		\item If $\bG$ is of type \tn{E}$_6$ then $A/\bG^F$ is a subgroup of $(C_3 \rtimes C_2) \times C_a \cong S_3 \times C_a$.
		\item Otherwise $A/\bG^F$ is a subgroup of $(V_4 \rtimes C_2) \times C_a \cong D_8 \times C_a$ or $(C_4 \rtimes C_2) \times C_a \cong D_8 \times C_a$.
	\end{itemize}
\end{corollary}

We conclude this section with results on the group theoretic structure of $A$. The following two lemmas are used in the proof of Proposition~\ref{prop1new}. 

\begin{lemma}\label{centraliser}
$\mathrm{C}_{A}(\bG^F) = \mathrm{Z}(\bG) \langle F \rangle$
\end{lemma}

\begin{proof}
It is clear that $\mathrm{Z}(\bG)  \langle F \rangle$ acts trivially on $\bG^F$, so $\mathrm{Z}(\bG) \langle F \rangle \subseteq \mathrm{C}_{A}(\bG^F) $. Now suppose that $g \in \mathcal{L}^{-1}(\tn{Z}(\bG ))$ and $d \in \langle F_0, \Gamma\rangle$ are such that $gd \in \mathrm{C}_{A}(\bG^F)$. As $gd$ induces the identity on $\bG^F$ it follows from \cite[Theorem 2.5.1]{G/L/SIII} that $g$ and $d$ both induce the identity on $\bG^F$. Therefore $g \in \tn{Z}(\mathcal{L}^{-1}(\tn{Z}(\bG)) = \tn{Z}(\bG)$ and $d = F^i$ for some $i$, so $gd \in \tn{Z}(\bG)\langle F \rangle$. 
\end{proof}

\begin{lemma}\label{orderF}
If $F$ acts trivially on $\mathrm{Z}(\bG)$ then $F$ has order $\mathrm{exp}(\mathrm{Z}(\bG))$ on $\mathcal{L}^{-1}(\tn{Z}(\bG ))$.
If $\mathrm{Z}(\bG)$ is cyclic and $F$ acts non-trivially on $\mathrm{Z}(\bG)$ then $F$ has order $2$ on $\mathcal{L}^{-1}(\tn{Z}(\bG ))$, and if $\mathrm{Z}(\bG)$ is not cyclic and $F$ acts non-trivially on $\mathrm{Z}(\bG)$ then $F$ has order $4$ on $\mathcal{L}^{-1}(\tn{Z}(\bG ))$.
\end{lemma}

\begin{proof}
Let $z \in \mathrm{Z}(\bG)$ and let $x \in \mathcal{L}^{-1}(\tn{Z}(\bG ))$ be such that $\mathcal{L}(x)=z$. Then $F(x)=x z$. If $F(z)=z$ then $F^k(x)=x z^k$ for all $k \in \N$, so $F^{\mathrm{o}(z)}(x)=x$ where $o(z)$ denotes the order of $z$. Therefore $F$ has order $\mathrm{exp}(\mathrm{Z}(\bG))$ on $\mathcal{L}^{-1}(\tn{Z}(\bG ))$.

Now suppose that $F$ is non-trivial on $\mathrm{Z}(\bG)$. If $\mathrm{Z}(\bG)=  \langle z \rangle $ then $F(z)=z^{-1}$ and therefore $F^2(x)=x F(z) z=x$. Thus $F$ has order 2 on $\mathcal{L}^{-1}(\tn{Z}(\bG ))$. If $\mathrm{Z}(\bG)$ is not cyclic then $\mathrm{Z}(\bG) \cong C_2 \times C_2$. Then for all $z \in \tn{Z}(\bG)$, $F^2(z)=z$ so $F^4(x)=F^3(z) F^2(z) F(z) z x= F(z)^2 z^2 x=x$. Hence $F$ has order $4$ on $\mathcal{L}^{-1}(\tn{Z}(\bG))$. 
\end{proof}

The following lemma is used in the proof of Proposition~\ref{prop:cocyclevalues}. 

\begin{lemma}\label{structure}
There exists an isomorphism $\tilde{\bG}^F/ \mathrm{Z}(\tilde{\bG}) \cong \mathcal{L}^{-1}(\mathrm{Z}(\bG)) /\mathrm{Z}(\bG^F)$ which maps $\bG^F \mathrm{Z}(\tilde{\bG}^F) / \mathrm{Z}(\tilde{\bG}^F)$ to $\bG^F \tn{Z}(\bG) / \mathrm{Z}(\bG).$
\end{lemma}

\begin{proof}
Let $\tilde g \in \wtGF$. Since $\wtG = \tn{Z}(\wtG) \bG$ we can choose a $z_{\tilde{g}} \in \mathrm{Z}(\wtG)$ such that $\tilde{g} z_{\tilde{g}} \in \bG$. Then $\mathcal{L}(\tilde g z_{\tilde g}) = \mathcal{L}(z_{\tilde g} ) \in \mathrm{Z}(\bG)$ so we can define a set-theoretic map $\wtGF \to \mathcal{L}^{-1}(\mathrm{Z}(\bG))$ given by $\tilde g \mapsto \tilde g z_{\tilde g}$. The coset $\tilde g z_{\tilde g} \tn{Z}(\bG)$ is uniquely determined by $\tilde g$, so the map $\wtGF \rightarrow \mathcal{L}^{-1}(\mathrm{Z}(\bG))/\tn{Z}(\bG)$ given by $\tilde g \mapsto \tilde g z_{\tilde g} \tn{Z}(\bG)$ is a group homomorphism. The kernel of this homomorphism is $\tn{Z}(\wtG)$, therefore $\tilde{\bG}^F/ \mathrm{Z}(\tilde{\bG}) \cong \mathcal{L}^{-1}(\mathrm{Z}(\bG)) /\mathrm{Z}(\bG^F)$. The second claim follows directly from the construction of the isomorphism.
\end{proof}

% ----------------------------------------------------------------
% ----------------------------------------------------------------
\subsection{Constructing projective representations}
\label{sec:construction}
% ----------------------------------------------------------------
% ----------------------------------------------------------------

In this section we construct projective representations of $A_\theta$ associated to $\theta$ for $ \theta \in \IBr(\bG^F)$. In Section~\ref{sec:fakegaloisactions} we will then use these projective representations to define a fake $m$th Galois action on $\IBr(\bG^F)$ first with respect to $\bG^F \lhd A$, and hence with respect to $\bG^F \lhd (\bG^F \rtimes \Aut(\bG^F))$, for every $m$ such that $(m, |\bG^F|) = 1$. 

The following technical lemma is central to the constructions which follow. Thanks to \cite[Lemma 4.5 (b)]{N/S/T}, we do not need to construct projective representations for $\theta \in \IBr(\bG^F)$ where $\mathrm{Out}(\bG^F)_\theta$ is cyclic (see the proof of Proposition~\ref{prop:nottypeA}). Because of the graph automorphism of order 3, when $\bG^F$ is of type $\tn{D}_4$ some cases require special methods. It is therefore often convenient to make one or both of the following assumptions. 

\begin{assumption}\label{assumption} 
Let $\theta \in \IBr(\bG^F)$. 
\begin{itemize}
\item[(i)] $\mathrm{Out}(\bG^F)_\theta$ is non-cyclic.
\item[(ii)] If $\bG^F$ is of type \tn{D}$_4$, then $A_\theta \leq \mathcal{L}^{-1}(\mathrm{Z}(\bG)) \rtimes \langle F_0, \gamma_\theta \rangle$ where $\gamma_\theta \in \Gamma$ is such that $\gamma_\theta^2 = 1$. 
\end{itemize}
\end{assumption}

\begin{notation}\label{not:gammadash}
In Lemma \ref{lem:U1U2}, Proposition~\ref{prop1new} and Proposition \ref{prop2new}, if $\bG^F$ is of type D$_4$ then we work under Assumption \ref{assumption} (ii). We let $\Gamma' := \langle \gamma_\theta \rangle$ if $\bG^F$ is of type D$_4$, let $\Gamma': = \Gamma$ otherwise. Then $A_\theta \leq \mathcal{L}^{-1}(\mathrm{Z}(\bG)) \rtimes \langle F_0, \Gamma' \rangle$ and $A_\theta / \bG^F$ is a subgroup of $(\tn{Z}(\bG) \rtimes \Gamma') \times \langle F_0' \rangle$ with $\Gamma' = \langle \gamma' \rangle \cong C_2$.
\end{notation}

\begin{lemma}\label{lem:U1U2}
Let $\theta \in \IBr(\bG^F)$ be a character satisfying Assumption \ref{assumption} (i) and (ii), and let $\gamma'$ be as in Notation \ref{not:gammadash}. Then there exist finite groups $U_1 \lhd A_{\theta}$ and $U_2 \leq A_{\theta}$ such that $A_\theta=U_1 U_2$, $U_1 \cap U_2= \bG^F$, $\mathrm{Z}(\bG) \leq U_1$, and either $F \in U_2$ or $\gamma' F \in U_2$ and $\gamma' \in U_1$. Furthermore, $U_1/ \bG^F$ is a subgroup of $D_8$ if $\bG$ is of type D, $U_1/ \bG^F$ is a subgroup of $S_3$ otherwise, and $U_2/\bG^F$ is cyclic.
\end{lemma}

\begin{proof}
Let $G:=\mathrm{Z}(\bG) \rtimes \Gamma'$ and let $H:= \langle F_0' \rangle$. Let $U$ be a subgroup of $G \times H$ such that $A_\theta/ \bG^F \cong U$. Recall that (as discussed in Section \ref{subsec:grouptheory}) $U$ has the form	
\[U= \{(g,h) \in G_1 \times H_1 \mid \varphi(g G_2)=h H_2 \} ,\]
for some $5$-tuple $(G_1,G_2,H_1,H_2,\varphi)$ where $G_2 \lhd G_1 \leq G$, $H_2 \lhd H_1 \leq H$, and $\varphi: G_1/G_2 \to H_1/H_2$ is an isomorphism. If $\bG$ is of type D then by Assumption \ref{assumption} (ii) and the definition of $\Gamma'$, $G$ is a subgroup of $D_8$. If $\bG$ is not of type D then $G$ is a subgroup of $S_3$. Thus there are only a small number of possibilities for $(G_1, G_2, H_1, H_2, \varphi)$, so we prove the claim by considering each of them individually. Note that $U \cap (G \times 1) = G_2 \times 1$ and $U \cap (1 \times H)=1 \times H_2$. Since $F$ and $\tn{Z}(\bG)$ act trivially on $\IBr(\bG^F)$,  both $F$ and $\tn{Z}(\bG)$ are contained in $A_\theta$ and hence $F$ and $\tn{Z}(\bG)\bG^F/\bG^F$ are contained in $A_\theta /\bG^F$.
It follows from the definition of $G$ that $\tn{Z}(\bG)\bG^F/\bG^F \leq  G \times 1 \cap U=G_2 \times 1$. Moreover, either $F \in H$ (in which case $F \in  1 \times H_2$), or $\gamma' F \in H$. 

First suppose that $G_2 = 1$. It then follows from Lemma~\ref{subgroupdirectproduct} that $U \cong H_1$, which is cyclic. Since $A_\theta / \bG^F$ surjects onto $\Out(\bG^F)_\theta$, therefore  $\Out(\bG^F)_\theta$ is cyclic, contradicting Assumption~\ref{assumption} (i).

Now suppose that $G_1 = G_2$. Then since $G_1/G_2 \cong H_1/H_2$, it follows that $H_1 = H_2$ and $\varphi(gG_2) = hH_2$ for all $ g \in G_1$, $h \in H_1$. Therefore $U =G_1 \times H_1$. 
Let $U_1, U_2 \leq A_\theta$ be such that $U_1/\bG^F \cong G_1 \times 1$ and $U_2 / \bG^F \cong 1 \times H_1$. Then $U_2/\bG^F$ is cyclic, and $U_1 /\bG^F$ is a subgroup of $D_8$ if $\bG$ is of type D, and a subgroup of $S_3$ if $\bG$ is not of type D. 
If $F \in H$ then $F \in 1 \times H_2 = 1 \times H_1$, as mentioned above, so $F \in U_2$. If $\gamma' F \in H$ then since $\gamma' \in G$, it follows that $\gamma' \in G_1$ and $\gamma' F \in H_1$ and hence $\gamma' F \in U_2$ and $\gamma' \in U_1$. As $\tn{Z}(\bG) \bG^F / \bG^F \leq G_2 \times 1$, clearly $\tn{Z}(\bG) \leq U_1$.   

By examination of the possibilities for $(G_1, G_2, H_1, H_2, \varphi)$, it only remains to consider the case when $\theta \in \IBr(\bG^F)$ is such that $G_1/G_2 \cong H_1/H_2 \cong C_2$. 

First suppose that $G_2$ has a direct complement in $G_1$, so $G_1 = G_2 \rtimes \langle \nu \rangle$ where $\nu^2 = 1$. Let $i \in \mathbb{N}$ such that $H_1=\langle (F_0')^i \rangle$ and $H_2=\langle (F_0')^{2i} \rangle$. We claim that $U= G_2 \rtimes \langle \nu (F'_0)^i \rangle$. Since $\langle \nu \rangle \cong G_1/G_2 \cong H_1/H_2 \cong \langle (F_0')^i H_2 \rangle$, the isomorphism $\varphi$ sends $\nu G_2$ to $(F_0')^iH_2$. Any $(g, h) \in U$ can be written as $(g_2, 1)(\nu^j, (F_0')^{ik})$ for some $g_2 \in G_2$ and some positive integers $j$ and $ k$. Then $\varphi(\nu^j G_2) = (F_0')^{ik} H_2$ and $\varphi(\nu^j G_2) = \varphi(\nu G_2)^j = (F_0')^{ij} H_2$ so $(F_0')^{i(k-j)} \in H_2$. Hence $\varphi(\nu^{k-j}) \in G_2 \cap \langle \nu \rangle = \{1 \}$ since $\langle \nu \rangle$ is a direct complement of $G_2$ in $G_1$. Therefore $\nu^j = \nu^k$ and $(g,h) = (g_2, 1)(\nu, (F_0')^i)^j \in G_2 \rtimes \langle \nu (F_0')^i \rangle$, proving the claim. 

Now we let $U_1, U_2 \leq A_\theta$ such that $U_1/ \bG^F=G_2 \times 1$ and $U_2/ \bG^F= 1 \times \langle \nu (F'_0)^i \rangle$. As in the previous case, it is immediate that $\tn{Z}(\bG) \leq U_1$, $U_2/ \bG^F$ is cyclic, and $U_1/ \bG^F$ is a subgroup of $D_8$ if $\bG$ is of type D, and a subgroup of $S_3$ otherwise. If $F \in H$ then $F \in 1 \times H_2 \leq 1 \times \langle \nu (F_0')^i \rangle$ since  $\nu^2 = 1$, and hence $F \in U_2$. Otherwise $\gamma' F \in H$ so $\gamma' F = (F_0')^r$ for some positive integer $r$. Then since $F \in U$, we have $\gamma' (F_0')^r \in G_2 \rtimes \langle \nu (F_0')^i \rangle$. If $\gamma' \in G_2$ then $\gamma' \neq \nu$ so $r$ is an even multiple of $i$, and hence $\gamma' \in U_1$ and $\gamma' F \in U_2$.   If $\gamma' = \nu$ then $F = \nu (F_0')^r \in \langle \nu (F_0')^i \rangle$, so $F \in U_2$. Finally, if $\gamma' = g_2 \nu$ for some nontrivial $g_2 \in G_2$, then note that $G_1 = G_2 \rtimes \langle g_2 \nu \rangle$ and by defining $U_1$ and $U_2$ now in relation to this new decomposition of $G_1$, again we see that $F \in U_2$. 

The final case to consider is when $G_1/G_2 \cong H_1/H_2 \cong C_2$ and $G_2$ does not have a direct complement in $G_1$. This occurs only if $\bG$ is of type D$_n$ with $G \cong D_8$ and $G_1 \cong C_4$, the unique cyclic subgroup of $D_8$ of order 4. If $n$ is even then let $z_1, z_2 \in \bG$ be such that $\tn{Z}(\bG) = \langle z_1, z_2 \rangle$. Then $G = (\langle z_1 \rangle \times \langle z_2 \rangle) \rtimes \langle \gamma' \rangle$. In this case fix $z_3 := z_1 z_2$. Then $G_1 = \langle z_1 \gamma'  \rangle$ or $\langle z_2 \gamma' \rangle $, and $G_2 = \langle z_3 \rangle$. If $n$ is odd then then let $z_1 \in \bG$ be such that $\tn{Z}(\bG) = \langle z_1\rangle$. Then $G = \langle z_1 \rangle \rtimes \langle \gamma' \rangle$ and $G_1 = \langle z_1 \rangle$. In this case we fix $z_3 := z_1^2$, so again we get $G_2= \langle z_3 \rangle$. 
In both cases we can show that $U=\langle z_3 \rangle \times \langle z_1  \gamma' (F_0')^i \rangle$, using arguments similar to the case above where $G_2$ has a direct complement in $G_1$. As usual, we let $U_1, U_2 \leq A_\theta$ such that $U_1/ \bG^F\cong\langle z_3 \rangle \times 1$ and $U_2/ \bG^F\cong 1 \times \langle z_1  \gamma' (F_0')^i \rangle$. Clearly $\tn{Z}(\bG) \leq U_1$, $U_2/ \bG^F$ is cyclic, and $U_1/ \bG^F$ is a subgroup of $D_8$. 
Since $F \in U$, either $F \in \langle z_1  \gamma' (F_0')^i \rangle$ or $z_3 F \in \langle z_1  \gamma' (F_0')^i \rangle$. In the first case then $F \in U_2$ so we are done. Suppose that $z_3 F \in \langle z_1  \gamma' (F_0')^i \rangle$. The map $\delta: A/\bG^F \twoheadrightarrow \Out(\bG^F)$ introduced at the beginning of Section~\ref{sec:notofTypeA} induces a surjection $U \twoheadrightarrow \Out(\bG^F)_\theta$, so $\Out(\bG^F)_\theta \leq \langle \delta_{z_2}, \delta_{z_1}  \gamma' (F_0')^i \rangle$. Now since $z_3 F \in \langle z_1  \gamma' (F_0')^i \rangle$, we have $\delta_{z_3} F \in \langle \delta_{z_1}  \gamma' (F_0')^i \rangle$ and therefore $\Out(\bG^F)_\theta $ is cyclic, contradicting Assumption~\ref{assumption} (i).
\end{proof}

\begin{proposition}\label{prop1new}
Suppose that $\bG$ is of type \tn{B, C, E}$_6$ or \tn{E}$_7$ and let $\theta \in \IBr(\bG^F)$ satisfying Assumption \ref{assumption} (i). 
There exists a projective representation $\mP$ of $A_\theta$ associated to $\theta$ with factor set $\alpha$ such that $\alpha^6=1$ and for every $c\in C_{A}(\bG^F)$, $\mP(c) = \xi Id$ for a sixth root of unity $\xi$.
\end{proposition}

\begin{proof}
By Lemma~\ref{lem:U1U2}, there exist finite groups $U_1$ and $U_2$ such that $\mathrm{Z}(\bG) \leq U_1 \lhd A_\theta$, $U_2 \leq A_\theta$, $A_\theta=U_1 U_2$, $U_1 \cap U_2= \bG^F$,  $U_2/\bG^F$ is cyclic and $U_1/ \bG^F$ is a subgroup of $S_3$, and either $F \in U_2$, or $F\gamma' \in U_2$ and $\gamma' \in U_1$.

Since $U_1/\bG^F$ has cyclic Sylow subgroups, $\theta$ extends to a character $\theta_1$ of $U_1$ by \cite[Theorem 8.29]{N}, and since $U_2/ \bG^F$ is cyclic $\theta$ also extends to a character $\theta_2$ of $U_2$. Let $\mD_1: U_1 \ra \tn{GL}_{\theta(1)}(\F)$ be a representation of $U_1$ affording $\theta_1$ and let $\mD_2: U_2 \ra \tn{GL}_{\theta(1)}(\F)$ be a representation of $U_2$ affording $\theta_2$ which agrees with $\mD_1$ on $\bG^F$. It then follows from Lemma~\ref{lem:projrepgen} that there exists a projective representation $\mP: A_\theta \ra \tn{GL}_{\theta(1)}(\F)$ associated to $\theta$ with factor set $\alpha$ such that for all $g = u_1u_2, g'= u_1'u_2' \in A_\theta$, $u_1, u_1' \in U_1$, $u_2, u_2' \in U_2$, 
\begin{align*}
\mP(g) & = \mD_1(u_1) \mD_2(u_2), \\
\mbox{and } \hspace{1ex} 
\alpha(g ,g' ) &  = \lambda_{u_2}( u_1'),
\end{align*}
where $\lambda_{u_2}$ denotes the linear Brauer character of $U_1/\bG^F$ such that $\theta_1^{u_2} \lam_{u_2} = \theta_1$. Since $U_1/\bG^F$ is a subgroup of $S_3$, $\lambda_{u_2}$ takes as values 2nd or 3rd roots of unity. In particular, $\alpha^6 = 1$.

Recall that $\mathrm{C}_{A}(\bG^F) \cong \mathrm{Z}(\bG) \langle F \rangle$ by Lemma \ref{centraliser}. For $c=z F^i  $ an arbitrary element of $\mathrm{C}_{A}(\bG^F)$, we have
\[\mP(c)=\alpha_j(z,F^i)^{-1} \mP(z) \mP(F^i) = \alpha_j(z,F^i)^{-1} \mD_1(z) \mP(F^i) ,\]
where the second equality holds because $\tn{Z}(\bG) \leq U_1$. Since $\bG$ is of type B, C, E$_6$ or E$_7$, $|\mathrm{Z}(\bG)| \leq 3$ and by Lemma~\ref{orderF}, $F$ has order less than or equal to $3$ on $\mL^{-1}(\mathrm{Z}(\bG))$. 

If $F \in U_2$ then $\mP(F^i)=\mD_2(F^i)$ and since $\mD_1(z)$ and $\mD_2(F^i)$ are matrices containing roots of unity of order less or equal than $3$, it follows that $\mP(c) = \xi Id$ for some $\xi$ such that $\xi^6 = 1$. If $\gamma' F \in U_2$ and $\gamma' \in U_1$ then $\mP(F)=\mD_1(\gamma') \mD_2(\gamma' F)$. Since $F \in \mathrm{C}_{A}(\bG^F)$, $\mP(F)$ is a scalar matrix and hence by a base change we can also assume that $\mD_2 (\gamma' F)$ is a diagonal matrix. Therefore $\mD_1( \gamma')=\mP(F) \mD_2(\gamma' F)^{-1}$ is also a scalar matrix. Since the orders of $ \gamma'$ and $ \gamma' F$ are divisors of $6$, the entries of $\mD_1(\gamma')$ and $\mD_2( \gamma' F)$ are $6$th roots of unity. Hence the scalar associated to $\mP(F)$, and therefore the scalar associated to $\mP(F^i)$, is a $6$th root of unity and so $\mP(c) = \xi Id$ for some sixth root of unity $\xi$. 
\end{proof}

\begin{proposition}\label{prop2new}
Suppose that $\bG$ is of type \tn{D} and suppose that $\theta \in \IBr(\bG^F)$ satisfies Assumption \ref{assumption} (i) and (ii). Then there exists a projective representation $\mP$ of $A_\theta$ associated to $\theta$ with factor set $\alpha$ such that $\alpha^4=1$ and for every $c\in C_{A}(\bG^F)$, $\mP(c) = \xi Id$ for a fourth root of unity $\xi$.
\end{proposition}

\begin{proof}
Lemma~\ref{lem:U1U2} shows that there exist finite groups $U_1$ and $U_2$ such that $\mathrm{Z}(\bG) \leq U_1 \lhd A_\theta$, $U_2 \leq A_\theta$, $A_\theta=U_1 U_2$, $U_1 \cap U_2= \bG^F$, $U_1/ \bG^F$ is a subgroup of $D_8$, $U_2/\bG^F$ is cyclic and either $F \in U_2$, or $F\gamma' \in U_2$ and $\gamma' \in U_1$. Unlike Proposition \ref{prop1new}, however, when $\bG$ is of type D we cannot assume that $\theta$ extends to $U_1$ as it is possible that $U_1/\bG^F$ has non-cyclic Sylow subgroups. We therefore need to construct $\mathcal \mP$ in two steps. 

We claim that there exist finite groups $K_1$ and $K_2$ such that $U_1=K_1 K_2$, $K_1 \lhd U_1$, $K_1 \cap K_2= \bG^F$, $K_2$ is normalised by $U_2$, and such that $K_1/ \bG^F$ is a subgroup of $C_4$ and $K_2/\bG^F$ is a subgroup of $C_2$. In addition, $\mathrm{Z}(\bG)$ is contained in either $K_1$ or $K_2$ and whenever $\gamma' \in U_1$, either $\gamma' \in K_1$ or $\gamma' \in K_2$. If $U_1/ \bG^F$ is cyclic then the claim holds trivially with $K_1 = U_1$ and $K_2 = \bG^F$. 

Suppose that $U_1/ \bG^F$ is not cyclic. If $U_1/ \bG^F$ is isomorphic to $D_8$ then $\gamma'$ is not a power of the field automorphism $F_0$, and hence $\bG^F$ is untwisted. In this case we let $ K_1$ and $K_2$ be subgroups of $U_1$ such that $K_1/ \bG^F$ is the unique cyclic  subgroup of $U_1/ \bG^F$ of order $4$ and $K_2/ \bG^F = \langle \gamma' \rangle$. If $n$ is even then since $\bG^F$ is untwisted, $\mathrm{Z}(\bG)= \mathrm{Z}(\bG)^F \leq \bG^F \leq K_1$. If $n$ is odd then $K_1/ \bG^F \cong \mathrm{Z}(\bG)$ and therefore $\mathrm{Z}(\bG) \leq K_1$. Hence the properties for $\tn{Z}(\bG)$ and $\gamma'$ are satisfied. 
Now suppose that $U_1/ \bG^F$ is isomorphic to $C_2 \times C_2$. If $\tn{Z}(\bG) \neq \tn{Z}(\bG)^F$ then we let $K_2= \mathrm{Z}(\bG) \bG^F$ and if $\tn{Z}(\bG) = \tn{Z}(\bG)^F$ then we let $K_2$ be such that $K_2/ \bG^F$ is a subgroup of $U_1/ \bG^F$ of order $2$ stabilised by the action of $U_2 / \bG^F$. If $\gamma' \in  U_1$ but $\gamma' \notin K_2$ we let $K_1/\bG^F= \langle \gamma' \rangle$. Otherwise, we let $K_1/ \bG^F$ be a direct complement of $K_2/ \bG^F$ in $U_1 / \bG^F$. Again, the properties for $\tn{Z}(\bG)$ and $\gamma'$ are clearly satisfied. Finally, since $U_1/ \bG^F$ is not cyclic, it follows from the proof of Lemma \ref{lem:U1U2} that (in the notation of Lemma \ref{lem:U1U2}) either $U_1 / \bG^F \cong G_1 = G_2$, $U_2 / \bG^F \cong H_1$ and $A_\theta/\bG^F \cong G_1 \times H_1$; or $G_1/G_2 \cong C_2$, $G_2$ has a direct complement in $G_1$, $U_1 / \bG^F \cong G_2$, $U_2 / \bG^F \cong \langle \nu (F_0')^i \rangle$ and $A_\theta/\bG^F \cong G_2 \times \langle \nu (F_0')^i\rangle$. In particular, $U_2/ \bG^F$ centralises $U_1/ \bG^F$ and hence $U_2$ centralises $K_2$ and the claim is proved.

We now construct $\mP$. Let $\mD$ be a representation of $\bG^F$ affording $\theta$. As $K_1/ \bG^F$, $K_2/\bG^F$ and $U_2/\bG^F$ are cyclic, $\mD$ extends to representations $\mD_{K_1}$ of $K_1$, $\mD_{K_2}$ of $K_2$ and $\mD_{U_2}$ of $U_2$. As $K_2$ is normalised by $U_2$, it follows from Lemma~\ref{lem:projrepgen} that there exists a projective representation $\mP_{K_2U_2}$ of $K_2U_2$ associated to $\theta$ with factor set $\alpha_{K_2U_2}$ such that 
\[
\hspace{6ex} \mP_{K_2U_2}(ku)  = \mD_{K_2}(k)\mD_{U_2}(u), \hspace{3ex} \mbox{and}\]
\vspace{-3ex}
\[ 
\alpha_{K_2U_2}(ku, k'u')  = \lambda_{u}(k'),
\]
\noindent
for any $ku, k'u' \in K_2U_2$, where $\lam_u \in \IBr(K_2/\bG^F)$. 
Applying Lemma~\ref{lem:projrepgen} again yields a projective representation $\mP$ of $A_\theta  = K_1K_2U_2$ associated to $\theta$ with factor set $\alpha$ such that 
\[
\hspace{6ex} \mP(k_1k_2u)   = \mD_{K_1}(k_1)\mP_{K_2U_2}(k_2u)
, \hspace{3ex} \mbox{and}
\]
\vspace{-2ex}
\[
 \alpha(k_1k_2u, k_1'k_2'u')  =  \lambda_{k_2u}(k_1')\alpha_{K_2U_2}(k_2u, k_2'u')
\]
\noindent
for any $k_1k_2u, k_1'k_2'u' \in K_1K_2U_2$, where $\lambda_{k_2u} \in \IBr(K_1/\bG^F)$. Since $K_1/\bG^F$ is a subgroup of $C_4$ and $K_2/\bG^F$ is a subgroup of $C_2$, it follows that $\lambda_{k_2u}(k_1')\lambda_{u}(k_2')$ has order dividing 4 and hence $\alpha^4 = 1$.

Finally, let $c = z F^i$ be an arbitrary element of $\mathrm{C}_{A}(\bG^F)$. Note that $\mP(c) = \alpha(z, F^i)^{-1} \mP(z) \mP(F^i)$. By construction, $\tn{Z}(\bG)$ is contained in either $K_1$ or $K_2$ so $\mP(z) = \mD_{K_1}(z)$ or $\mP(z)  = \mD_{K_2}(z)$. Because $\bG$ is of type D, the central element $z$ has order dividing 4 and hence the scalar associated to the scalar matrix $\mP(z)$ is a fourth root of unity. It follows from Lemma~\ref{lem:actiononZG} that the order of $F$ on $\mL^{-1}(\mathrm{Z}(\bG))$ divides 4. 
If $F \in U_2$ then $\mP(F^i) = \mD_{U_2}(F^i)$, a scalar matrix associated to a fourth root of unity.
Otherwise, $\gamma' F \in U_2$ and $\gamma' \in U_1$ and hence, as in Proposition~\ref{prop1new}, $\mP(F) = \mD_{K_1} (\gamma') \mD_{U_2}(\gamma' F)$ or $ \mD_{K_2} (\gamma') \mD_{U_2}(\gamma' F)$. As in Proposition~\ref{prop1new}, we can assume that $\mD_{U_2}(\gamma' F)$ is a diagonal matrix and since the orders of $\gamma'$ and $\gamma' F$ divide 4, the scalar associated to $\mP(F)$, and therefore the scalar associated to $\mP(F^i)$, is a fourth root of unity. In particular, $\mP(c) = \xi Id$ for some fourth root of unity $\xi$. 
\end{proof}

\begin{proposition}\label{prop:cocyclevalues}
Suppose that $\bG^F$ is of type $\tn{D}_4$ and suppose that $\theta \in \IBr(\bG^F)$ satisfies Assumption \ref{assumption} (i) but not Assumption \ref{assumption} (ii). Then there exists an integer $r \in \{4, 6\}$ and a projective representation $\mP$ of $A_\theta$ associated to $\theta$ with factor set $\alpha$ such that $\alpha^r=1$ and for every $c\in C_A(\bG^F)$, $\mP(c) = \xi Id$ for an $r$th root of unity $\xi$.
\end{proposition}

\begin{proof}
Recall that by Corollary~\ref{cor:AoverGF}, $A_\theta / \bG^F$ is isomorphic to a subgroup of 
$ S_4 \times C_a $ for some positive integer $a$. As $\bG^F$ is of type D$_4$ the automorphism $F$ acts trivially on $\mathrm{Z}(\bG)$, see for instance the proof of Lemma \ref{lem:actiononZG}. Therefore, $\mathrm{Z}(\bG) = \mathrm{Z}(\bG)^F =  \mathrm{Z}(\bG^F)$ by \cite[Proposition 3.6.8]{C}.

Let $\nu \in \IBr(\tn{Z}(\bG))$ be such that $\theta_{\mathrm{Z}(\mathbf{G})_{\ell'}}=\theta(1) \nu$. Note that $\mathcal{L}^{-1}(\mathrm{Z}(\bG))$ and $F_0$ act trivially on $\mathrm{Z}(\bG)$, so in fact only $\Gamma$ acts non-trivially on $\mathrm{Z}(\bG)$. Using the parametrisation of generators of $\mathrm{Z}(\bG)$ given in \cite[Table 1.12.6]{G/L/SIII}  we see that any graph automorphism of order 3 acts non-trivially on $\mathrm{Z}(\bG)$. However, the only element of $\IBr(\mathrm{Z}(\bG))$ stabilised by an action of order 3 is the trivial character. Thus, if $\nu$ is nontrivial then $A_\theta \leq \mathcal{L}^{-1}(\mathrm{Z}(\bG)) \rtimes \langle F_0, \gamma \rangle$ for some suitable graph automorphism $\gamma \in \Gamma$ of order $2$. In other words, then $\theta$ satisfies Assumption \ref{assumption} (ii). We can therefore assume that $\theta$ is trivial on $\tn{Z}(\bG)$. 

Note that $\bG^F  \langle F \rangle \lhd A_\theta$. We can extend $\theta$ to a character $\theta_0 = \theta \times 1_{\langle F \rangle} \in \IBr(\bG^F \langle F \rangle)$, and since $\theta$ is trivial on $\mathrm{Z}(\bG)$ we have  $A_\theta=A_{\theta_0}$. Since $A_\theta / (\bG^F  \langle F \rangle)$ is isomorphic to a subgroup of $S_4 \times C_a$, Corollary \ref{cor:structureofsylows} shows that $A_{\theta}/ (\bG^F  \langle F \rangle) \cong G_2 \rtimes H_1$ for some subgroup $G_2$ of $S_4$ and some cyclic group $H_1$. If $G_2$ is isomorphic to a subgroup of $S_3$ or $D_8$ then the existence of a projective representation $\mP$ associated to $\theta$ with factor set $\alpha$ satisfying $\alpha^6=1$ or $\alpha^4=1$ (respectively) follows from the arguments in Propositions~\ref{prop1new} and \ref{prop2new}. 

It remains to consider the case that $A_{\theta}/ (\bG^F  \langle F \rangle) \cong G_2 \rtimes H_1$ and $G_2 \cong A_4$ or $S_4$. As shown in Corollary \ref{cor:structureofsylows}, then $V_4 \times 1 \lhd A_\theta/(\bG^F \langle F \rangle)$. Thus $\mathcal{L}^{-1}(\mathrm{Z}(\bG)) \leq A_\theta$, so $\theta$ is stable in $\mathcal{L}^{-1}(\mathrm{Z}(\bG))$. 
We claim that $\theta$ extends to $\mathcal{L}^{-1}(\mathrm{Z}(\bG))$ and hence to a character of $\mathcal{L}^{-1}(\tn{Z}(\bG))\langle F \rangle$ lying over $\theta_0$. If $\ell=2$ then the quotient $\mathcal{L}^{-1}(\mathrm{Z}(\bG))/ \bG^F$ is an $\ell$-group and hence $\theta$ extends to a character $\hat \theta$ of $\mathcal{L}^{-1}(\mathrm{Z}(\bG))$ by Green's indecomposability theorem \cite[Theorem 8.11]{N}. Suppose now that $\ell \neq 2$. As $\mathcal{L}^{-1}(\mathrm{Z}(\bG)) \leq A_\theta$ it follows that the character $\theta$ is stable under diagonal automorphisms, and hence $\theta$ is $\wtGF\!$-stable. In particular, it follows from \cite[Theorem B]{G3} that $\theta$ extends to a character $\tilde{\theta}$ of $\tilde{\bG}^F$ such that $\tilde{\theta}$ also lies over the trivial character of $\mathrm{Z}(\tilde{\bG}^F)$. But we saw in Lemma \ref{structure} that $\tilde{\bG}^F/ \mathrm{Z}(\tilde{\bG}^F) \cong \mathcal{L}^{-1}(\mathrm{Z}(\bG)) /\mathrm{Z}(\bG^F)$, therefore $\theta$ extends to a character $\hat \theta$ of $\mathcal{L}^{-1}(\mathrm{Z}(\bG))$. Note that $F$ acts trivially on the quotient $\mathcal{L}^{-1}(\mathrm{Z}( \bG ) / \mathrm{Z}( \bG )$. Therefore, the extension $\hat{\theta} \in \IBr(\mathcal{L}^{-1}(\mathrm{Z}(\bG)))/ \mathrm{Z}(\bG))$ is $F$-stable and hence can be extended to a character $\hat{\theta}_0$ of $\mathcal{L}^{-1}(\mathrm{Z}(\bG))\langle F \rangle$ lying above $\theta_0$, as claimed. 

We can now construct a suitable projective representation $\mP$ of $A_\theta = A_{\theta_0}$ which is associated to $\theta_0$, and therefore to $\theta$.  Since $A_{\theta} / \bG^F \langle F \rangle $ is isomorphic to a subgroup of $ V_4 \rtimes (S_3 \rtimes C_a)$, there exist groups $H_1 = \mathcal{L}^{-1}(\tn{Z}(\bG)) \langle F \rangle  \lhd A_\theta$ and $H_2 \leq A_\theta$ such that $H_1 H_2=A_ \theta$ and $H_1 \cap H_2=\bG^F \langle F \rangle $ with $H_1/\bG^F \langle F \rangle  \cong V_4$ and $H_2/\bG^F \langle F \rangle $ isomorphic to subgroup of $S_3 \times C_a$. Furthermore, there exist groups $K_1 \lhd H_2$ and $K_2 \leq H_2$ such that $K_1 K_2=H_2$ and $K_1 \cap K_2= \bG^F \langle F \rangle$ with $K_1/\bG^F \langle F \rangle $ isomorphic to a subgroup of $S_3$ and $K_2/\bG^F \langle F \rangle $ cyclic. As $K_1/\bG^F \langle F \rangle$ has cyclic Sylow subgroups, $\theta_0$ extends to $K_1$, so by Lemma \ref{lem:projrepgen} there exists a projective representation $\mathcal{P}_2$ of $H_2$ associated to $\theta$ with factor set $\alpha_2$ such that $\alpha_2^r=1$ for some $r \in \{2, 3\}$. By the claim above, $\theta_0$ extends to a character $\hat \theta_0$ of $H_1$. Let $\mathcal{D}_1$ be an ordinary representation of $H_1$ affording $\hat{\theta}_0$ such that $\mathcal{D}_1$ and $\mathcal{P}_2$ agree on $\bG^F \langle F \rangle $. Then by Lemma \ref{lem:projrepgen} again, there exists a projective representation $\mathcal{P}$ associated to $\theta_0$ with factor set $\alpha$ such that $\alpha^6=1$. Since $\theta_0$ extends $\theta$, $\mP$ is also associated to $\theta$. 

Finally, since $\tn{Z}(\bG) = \tn{Z}(\bG)^F$ Lemma \ref{centraliser} shows that $ \mathrm{C}_{A}(\bG^F) = \tn{Z}(\bG)\langle F \rangle \leq \bG^F \langle F \rangle$. Hence for any $c=z F^i \in \mathrm{C}_{A}(\bG^F)$ and any projective representation $\mP'$ of $A_\theta$ associated to $\theta_0$ (in particular, for $\mP$), we have
\[\mP'(c)=\theta_0(z F^i) I_{\theta(1)}= \theta(z),\]
so $\mP'(c) = \xi Id$ for a root of unity $\xi$ of order at most 2.
\end{proof}

% ----------------------------------------------------------------
% ----------------------------------------------------------------
\subsection{Fake Galois actions}
\label{sec:fakegaloisactions}
% ----------------------------------------------------------------
% ----------------------------------------------------------------

\noindent For $\theta \in \IBr(\bG^F)$, we let $\overline{\theta} \in \IBr(\bG^F)$ be the character given by $\overline{\theta}(g) = \theta(g^{-1})$ for all $g \in \bG^F$.

\begin{proposition}\label{prop:nottypeA}
Let $\bG$ be a simple, simply connected algebraic group of type \tn{B, C, D, E}$_6$ or \tn{E}$_7$. For any positive integer $m$ such that $(m, |\bG^F|) = 1$, there exists a fake $m$th Galois action on $\IBr(\bG^F)$ with respect to $\bG^F \lhd \bG^F \rtimes \Aut(\bG^F)$.  
\end{proposition}

\begin{proof}
	Let $\theta \in \IBr(\bG^F)$. We fix a positive integer $r_\theta$ as follows. If $\Out(\bG^F)_\theta$ is cyclic then $r_\theta := \exp(\tn{Z}(\bG)) $. If $\Out(\bG^F)_\theta$ is not cyclic then it follows from Propositions~\ref{prop1new}, \ref{prop2new} and \ref{prop:cocyclevalues} that there exists a projective representation $\mP$ of $A_{\theta}$ associated to $\theta$ with factor set $\alpha$ such that $\alpha^{k} = 1$ for some $k \in \{4, 6\}$. In this case we fix $r_\theta := k$.  

	Since $(m, |\bG^F|) = 1$ and $|\bG^F|$ is divisible by $2$ and $3$ by \cite[Table 24.1]{M/T}, $m$ and $r_\theta$ are coprime and hence $m \equiv \pm 1 \mathrm{ \, mod \, } (r_\theta)$. We can therefore define a map on $\IBr(\bG^F)$ by
	\[f_m(\theta) := \begin{cases} \theta &\mbox{if } m \equiv 1 \mathrm{\, mod \,} (r_\theta) \\
	\overline{\theta} & \mbox{if } m \equiv -1  \mathrm{\, mod \, } (r_\theta) .\end{cases}\]
Since $A_{\bar \theta} = A_{\theta}$ and $f_m(\bar \theta) = \overline{f_m(\theta)}$, the map $f_m : \IBr(\bG^F) \rightarrow \IBr(\bG^F)$ is a well defined bijection. We claim that this bijection defines a fake $m$th Galois action on $\IBr(\bG^F)$ with respect to $\bG^F \lhd \bG^F \rtimes \Aut(\bG^F)$.
	
	First suppose that $\mathrm{Out}(\bG^F)_\theta$ is cyclic and let $\nu \in \mathrm{IBr}( \mathrm{Z}(\bG^F) \mid \theta)$. Then 
	\[\nu^m  = \begin{cases} \nu &\mbox{if } m \equiv 1 \mathrm{\, mod \,} (r_\theta) \\
	\nu^{-1} & \mbox{if } m \equiv -1  \mathrm{\, mod \, } (r_\theta), \end{cases}\]
so $\nu^m \in \mathrm{IBr}( \mathrm{Z}(\bG^F) \mid f_m(\theta))$.  Thus it follows from \cite[Lemma 4.5(b)]{N/S/T} that 
	\[ (\bG^F \rtimes \Aut(\bG^F)_\theta,\bG^F,\theta)^{(m)} \approx (\bG^F \rtimes \Aut(\bG^F)_\theta,\bG^F,f_m(\theta)).\]
	
	Now suppose that $\mathrm{Out}(\bG^F)_\theta$ is not cyclic. By definition of $r_\theta$, there exists a projective representation $\mP$ of $A_{\theta}$ associated to $\theta$ with factor set $\alpha$ such that $\alpha^{r_\theta} = 1$. Moreover, by Propositions \ref{prop1new}, \ref{prop2new} and \ref{prop:cocyclevalues}, for every $c\in C_{A}(\bG^F)$, we have $\mP(c) = \xi Id$ where $\xi$ is an $r_\theta$th root of unity. It then follows from \cite[Proposition 4.7]{N/S/T} that 
	\[(A_{\theta}, \bG^F, \theta)^{(m)} \approx (A_{\theta}, \bG^F, f_m(\theta)).\]
	By construction of $A$, there exists a surjective map $A \twoheadrightarrow \bG^F \rtimes \Aut(\bG^F)$ such that $A/C_{A}(\bG^F) \cong \mathrm{Out}(\bG^F)$, and hence there exists a surjective map $A_\theta \twoheadrightarrow \bG^F \rtimes \Aut(\bG^F)_\theta$ such that $A_\theta/C_{A}(\bG^F) \cong \mathrm{Out}(\bG^F)_\theta$. Thus by \cite[Corollary 4.12]{S/V}, 
	\[(\bG^F \rtimes \mathrm{Aut}(\bG^F)_\theta, \bG^F, \theta)^{(m)} \approx (\bG^F \rtimes \mathrm{Aut}(\bG^F)_\theta, \bG^F, f_m(\theta)).\]
	
	Consequently, by Definition \ref{def:fakegalaction}, there exists a fake $m$th Galois action on $\IBr(\bG^F)$ with respect to $\bG^F \lhd \bG^F \rtimes \Aut(\bG^F)$.
\end{proof}

% ----------------------------------------------------------------
% ----------------------------------------------------------------
\section{Exceptional Covering Groups}\label{sec:except}
% ----------------------------------------------------------------
% ----------------------------------------------------------------
In this section we deal with the simple finite groups of Lie type with non-cyclic outer automorphism group whose Schur multiplier has non-trivial exceptional part. We continue to use the notation of Section \ref{sec:fglt} with $\bG$ a simple simply connected algebraic group defined over $\overline \F_p$ where $p \neq \ell$, $F: \bG \ra \bG$ a Frobenius endomorphism, and $\bG^F$ the finite group of the fixed points of $\bG$ under $F$. Let $S \cong \bG^F / \tn{Z}(\bG^F)$ be a finite simple group of Lie type, and let $X$ be the universal cover of $S$. Let $M(S)$ denote the Schur multiplier of $S$ and recall that  $\tn{Z}(X) \cong M(S) \cong M_c(S) \times M_e(S)$, where $M_c(S)$ denotes the canonical part of the Schur multiplier and $M_e(S)$ denotes the exceptional part.

\begin{lemma}\label{centre}
	Suppose that $S \in \{ O^+_8(2),U_6(2),{}^2 E_6(2), U_4(3) \}$. Let $\nu \in \mathrm{IBr}( \mathrm{Z}(X))$. Then either $\nu$ is trivial on $M_e(S)$ or $\mathrm{Out}(S)_{\nu}$ is cyclic.
\end{lemma}

\begin{proof}
	Write $\nu=\nu_1 \times \nu_2$ with $\nu_1 \in \mathrm{IBr}(M_c(S))$ and $\nu_2 \in \mathrm{IBr}(M_e(S))$ and suppose that $\nu_2$ is non-trivial. If $ S \in  \{O^+_8(2),U_6(2),{}^2 E_6(2) \}$ then by \cite[Table 6.3.1]{G/L/SIII}, the outer automorphism group $\mathrm{Out}(S) \cong S_3$ acts faithfully on $M_e(S) \cong C_2 \times C_2$. However, no non-trivial character of $C_2 \times C_2$ is stabilized by a non-trivial action of an automorphism of order $3$, therefore $\Out(S)_\nu \cong C_2$ is cyclic. If $S=U_4(3)$ then $\mathrm{Out}(S) \cong D_8$ acts faithfully on $M_e(S) \cong C_3 \times C_3$. Again, no non-trivial character of $C_3 \times C_3$ is stabilized by $C_2 \times C_2$ acting faithfully on $C_3 \times C_3$, hence $\Out(S)_\nu \cong C_4$ is cyclic.
\end{proof}

\begin{remark}\label{rmk:hatP}
Let $\pi: X \rightarrow \bG^F$ be the quotient map. Then there is a natural bijection 
\[\mathrm{IBr}( \bG^F) \rightarrow \mathrm{IBr}(X \mid 1_{M_e(S)}),\]
so we identify any $\theta \in  \IBr(\bG^F)$ with a character of $X$ which is trivial on $M_e(S)$, which we also denote by $\theta$. By \cite[Corollary B.8]{N2} there exists an isomorphism 
\[ \Aut(\bG^F) \rightarrow \Aut(X)_{M_e(S)} \]
sending $\phi \mapsto \hat \phi$ where $\phi \pi = \pi \hat \phi$. This induces a quotient map 
\[ \pi: X \rtimes \Aut(X) \rightarrow \bG^F \rtimes \Aut(\bG^F) \]
such that $\pi(X) = \bG^F$. Let $\theta \in \IBr(\bG^F)$ and let $\mP$ be a projective representation of $\bG^F \rtimes \Aut(\bG^F)_\theta$ associated to $\theta$ with factor set $\alpha$. Let $\hat \mP$ denote the projective representation of $X \rtimes \Aut(X)_\theta$ given by 
\[ \hat \mP (x) = \mP(\pi(x))\]
for all $x \in X \rtimes \Aut(X)_\theta$. Then the factor set $\hat \alpha$ of $\hat \mP$ satisfies $\hat \alpha(x, y) = \alpha(\pi(x), \pi(y))$ for all $x, y \in X \rtimes \Aut(X)_\theta$.  
\end{remark}

We set up some final pieces of notation. Let $G$ be a finite group and let $\mP: G \rightarrow \tn{GL}_n(\F)$ be a projective representation with factor set $\alpha: G \times G \rightarrow \F^{\times}$. We denote by $\overline \mP$ the projective representation given by $\overline \mP (g) = (\mP(g)^T)^{-1}$ for all $g \in G$. Its factor set $\overline \alpha$ satisfies $\overline \alpha (g, g') = \alpha(g, g')^{-1}$ for all $g, g' \in G$. Let $\sigma: \F \to \F$ denote the Galois automorphism given by $x \mapsto x^\ell$ for $ x\in \F$, and recall that if $\theta \in \IBr(G)$ then $\theta^{\sigma} \in \IBr(G)$, see \cite[Problem 2.10]{N}. We denote by $\mP^\sigma$ the projective representation given by $\mP^\sigma(g) = \sigma(\mP(g))$ for all $g \in G$. The associated factor set $\alpha^\sigma$ satisfies $\alpha^\sigma(g, g') = \sigma(\alpha(g, g'))$ for all $g, g' \in G$.

\begin{proposition}\label{prop:except1}
	Suppose that $S \in \{O^+_8(2),U_6(2),{}^2 E_6(2), U_4(3) \}$ and let $X$ be the universal cover of $S$. Then for any positive integer $m$ such that $(m,|X|)=1$, there exists a fake $m$th Galois action on $\IBr(X)$ with respect to $X \lhd X \rtimes \Aut(X)$.
\end{proposition}

\begin{proof}
	Let $\mathfrak{X}:= \mathrm{IBr}(X) \setminus \mathrm{IBr}(X \mid 1_{M_e(S)})$ and fix a positive integer $m$ such that $(m,|X|)=1$. We claim that there exist suitable bijections $f_m : \IBr(X \mid 1_{M_e(S)}) \rightarrow \IBr(X \mid 1_{M_e(S)}) $ and $g_m: \mathfrak X \rightarrow \mathfrak X$, such that the bijection $h_m : \IBr(X) \rightarrow \IBr(X)$ given by 
	\[h_m(\theta):= \begin{cases} 
	f_m(\theta)&\mbox{if } \theta \in  \IBr(X\mid 1_{M_e(S)}) \\
	g_m(\theta) & \mbox{if }\theta \in  \mathfrak X ,\end{cases}\]
	defines a fake $m$th Galois action on $\IBr(X)$ with respect to $X \lhd X \rtimes \Aut(X)$.

	Let $f_m: \mathrm{IBr}( \bG^F)  \to \mathrm{IBr}(\bG^F)$ be the bijection given in Corollary \ref{coro:typeA} (if $S = U_6(2)$ or $U_4(3)$) or Proposition \ref{prop:nottypeA} (if $S = O^+_8(2)$ or $^2\!E_6(2)$) such that 
	\[(\bG^F \rtimes \mathrm{Aut}(\bG^F)_\theta,\bG^F,\theta)^{(m)} \approx (\bG^F \rtimes \mathrm{Aut}(\bG^F)_\theta,\bG^F,f_m(\theta))\]
	for every $\theta \in \IBr(\bG^F)$. By Remark \ref{rmk:hatP}, it follows that
	\[((X \rtimes \Aut(X))_\theta,X,\theta)^{(m)} \approx ((X \rtimes \Aut(X))_\theta,X,f_m(\theta))\]
	for all $\theta \in \IBr(\bG^F)$.
	
	We will now construct $g_m$. Suppose first that $S \in  \{O^+_8(2),U_6(2),\,^2\! E_6(2) \}$ or $S = U_4(3)$ and $\ell = 2$. 
	Define a map $g_m: \mathfrak X \rightarrow \mathfrak X$ by 
	\[g_m(\theta) = \begin{cases} 
	\theta &\mbox{if } m \equiv 1 \mathrm{\, mod \, } 3 \\
	\overline{\theta} & \mbox{if } m \nequiv 1 \mathrm{\, mod \, } 3. \end{cases}\]
	This is an $\mathrm{Aut}(X)$-equivariant bijection. Let $\nu \in \mathrm{IBr}( \mathrm{Z}(X) \mid \theta)$. Note that if $S \in  \{O^+_8(2),U_6(2),\,^2\! E_6(2) \}$ then $\mathrm{Z}(X)$ is a subgroup of $C_2 \times C_2 \times C_3$ and if $S = U_4(3)$ and $\ell = 2$, then $\tn{Z}(X)_{\ell'} = C_3^2$. Therefore 
	\[\nu^m= \begin{cases} 
	\nu &\mbox{if } m \equiv 1 \mathrm{\, mod \, } 3 \\
	\nu^{-1} & \mbox{if } m \nequiv 1 \mathrm{\, mod \, } 3, \end{cases}\]
	and hence $\nu^m \in \mathrm{IBr}( \mathrm{Z}(X) \mid g_m(\theta))$. Since $\ell$ divides the order of $S$ and $p \neq \ell$, it only remains to define $g_m$ in the case where $S=U_4(3)$ and $\ell \in \{5, 7 \}$. If $\ell = 5$ we fix a pair of integers $(k_1, k_2) := (3,4)$ and if $\ell = 7$ we fix $(k_1, k_2) := (4,3)$. We then define a bijection $g_m : \mathfrak X \rightarrow \mathfrak X$ by
	\[g_m(\theta) = \begin{cases} 
	\theta &\mbox{if } m \equiv 1 \mathrm{\, mod \, } k_1 \text{ and } m \equiv 1 \mathrm{\, mod \, }  k_2 \\
	\theta^\sigma &\mbox{if } m \equiv -1 \mathrm{\, mod \, } k_1 \text{ and } m \equiv 1 \mathrm{\, mod \, }  k_2 \\
	\overline{\theta} &\mbox{if } m \equiv 1 \mathrm{\, mod \, }  k_1 \text{ and } m \equiv -1 \mathrm{\, mod \, } k_2 \\
	\overline{\theta^\sigma} & \mbox{if }m \equiv -1 \mathrm{\, mod \, } k_1 \text{ and } m \equiv -1 \mathrm{\, mod \, } k_2. \end{cases}\]
	Again, clearly $g_m$ is an $\mathrm{Aut}(X)$-equivariant bijection,  and if $\nu \in \mathrm{IBr}( \mathrm{Z}(X) \mid \theta)$ then $\nu^m \in \mathrm{IBr}( \mathrm{Z}(X) \mid g_m(\theta))$. Now since $\Out(S)_\theta$ is cyclic for all $\theta \in \mathfrak X$ by Lemma \ref{centre}, it follows from \cite[Lemma 4.5(b)]{N/S/T} and \cite[Corollary 4.12]{S/V} that 
	\[ ((X \rtimes\mathrm{Aut}(X))_\theta,X,\theta)^{(m)} \approx ((X \rtimes \mathrm{Aut}(X))_\theta,X,g_m(\theta))\]
	for all $\theta \in \mathfrak X$.
	
	Finally, with $h_m: \IBr(X) \rightarrow \IBr(X)$ defined as above, we have
	\[ ((X \rtimes\mathrm{Aut}(X))_\theta,X,\theta)^{(m)} \approx ((X \rtimes \mathrm{Aut}(X))_\theta,X,h_m(\theta)),\]
	for all $\theta \in \IBr(X)$, and hence $h_m$ defines a fake $m$th Galois action on $\IBr(X)$ with respect to $X \lhd X \times \Aut(X)$. 
\end{proof}

We now set up some notation for the following proposition. Let $\bG = \tn{SL}_3(\overline{\F}_4)$,  $\tilde{\bG}=\mathrm{GL}_3(\overline{\mathbb{F}}_4)$, and let $F_0: \tilde{\bG} \to \tilde{\bG}$ be the field automorphism given by squaring all matrix entries. Fix a Frobenius automorphism $F: \tilde \bG \rightarrow \tilde \bG$ given by $F:=F_0^2$. Then $\bG^F/\mathrm{Z}(\bG^F) = L_3(4)$ which we denote by $S$. Let $X$ be the universal cover of $S$. We denote the graph automorphism by $\gamma: \bG \to  \bG, \, (a_{ij}) \mapsto (a_{ij})^{-T}$. Let $\tilde{T}$ be the set of diagonal matrices in $\tilde \bG^F$ and note that $\tilde{T} \cong C_3^3$ is an elementary abelian group of exponent $3$. Let $t \in \tilde T$. It follows from \cite[Proposition 1.5]{Le} and \cite[Lemma 1.3]{D/L/M} that the map $\delta: \bG^F \to \bG^F$ given by $x \mapsto {}^t\!x$ realizes a diagonal automorphism of order $3$. Thus the order of $\delta \in \mathrm{Aut}(\bG^F)$ coincides with the order of $\delta \mathrm{Inn}(\bG^F)$ in $\mathrm{Out}(\bG^F)$.  

Let $A:= \langle F_0, \gamma, \delta \rangle$ considered as automorphisms of $S$ and let $\tilde{A}$ be the preimage of $A$ under the isomorphism $\mathrm{Aut}(X) \to \mathrm{Aut}(S)$ given in \cite[Corollary B.8]{N2}. Since $F_0(\delta) = \gamma (\delta) = \delta^{-1}$, it follows that $A \cong \Aut(S)/\Inn(S)$. Therefore $\tilde A \cong \Aut(X)/\Inn(X)$ and hence $\mathrm{C}_{X \rtimes \tilde{A}}(X) =  \mathrm{Z}(X)$.

\begin{proposition}\label{prop:except2}
	Let $S=L_3(4)$ and let $X$ be the universal cover of $S$. Then for any positive integer $m$ such that $(m,|X|)=1$, there exists a fake $m$th Galois action on $\IBr(X)$ with respect to $X \lhd X \rtimes \Aut(X)$.
\end{proposition}

\begin{proof}
	Since $\ell$ divides the order of $S$ and $p \neq \ell$, we can assume that  $\ell \in \{3,5,7\}$. If $\ell\in \{ 3,5 \}$ we fix a pair of integers $(k_1,k_2):=(3,4)$ and if $\ell=7$ we fix $(k_1,k_2):=(4,3)$.  Define a map $f_m: \IBr(X) \rightarrow \IBr(X)$ by
	\[f_m(\theta) = \begin{cases}
	\theta &\mbox{if } m \equiv 1 \mathrm{\, mod \, }  k_1 \text{ and } m \equiv 1 \mathrm{\, mod \, }  k_2 \\
	\theta^\sigma &\mbox{if }m \equiv -1 \mathrm{\, mod \, }  k_1 \text{ and } m \equiv 1 \mathrm{\, mod \, }  k_2  \\
	\overline{\theta} &\mbox{if } m \equiv 1 \mathrm{\, mod \, }  k_1 \text{ and } m \equiv -1 \mathrm{\, mod \, }  k_2  \\
	\overline{\theta^\sigma} & \mbox{if }m \equiv -1 \mathrm{\, mod \, }  k_1 \text{ and } m \equiv -1 \mathrm{\, mod \, }  k_2 . 
	\end{cases}\]
	This is clearly an $\Aut(X)$-invariant bijection.
	
	Let $\theta \in \IBr(X)$. Since $(X \rtimes \tilde{A})/\tilde{A} \cong \tilde{A} \cong \mathrm{Out}(S) \cong S_3 \times C_2$, it follows from Lemma \ref{lem:projrepgen} that there exists a projective representation $\mathcal{P}$ of $(X \rtimes \tilde{A})_{\theta}$ associated to $\theta$ with factor set $\alpha$ such that $\alpha^2=1$. The projective representations $\mathcal{P}^\sigma$, $\overline{\mathcal{P}}$ and $\overline{\mathcal{P}^\sigma}$ associated to $\theta^\sigma$, $\overline{\theta}$ and $\overline{\theta^\sigma}$ respectively, then also have factor set $\alpha$, and hence the second condition of Definition \ref{def:squiggleapprox} is satisfied. Now since $\mathrm{C}_{X \rtimes \tilde{A}}(X) =  \mathrm{Z}(X) \cong C_3 \times C_4^2$, one can check, as in the previous proposition, that if $\nu \in \mathrm{IBr}( \mathrm{Z}(X) \mid \theta)$  then $\nu^m \in \mathrm{IBr}( \mathrm{Z}(X) \mid f_m(\theta))$ and hence the first condition of  Definition \ref{def:squiggleapprox} is also satisfied. Thus, for every $\theta \in \IBr(X)$, 
	\[ ((X \rtimes \tilde{A})_\theta,X,\theta)^{(m)} \approx ((X \rtimes \tilde{A})_\theta,X,f_m(\theta)).\]
	
	Since there exists a surjective map $X \rtimes \tilde A \twoheadrightarrow X \rtimes \Aut(X)$ such that $(X \rtimes \tilde A)/C_{X \rtimes \tilde A}(X) \cong \mathrm{Out}(X)$, \cite[Corollary 4.12]{S/V} implies that 
	\[ ((X \rtimes \Aut(X))_\theta,X,\theta)^{(m)} \approx ((X \rtimes \Aut(X))_\theta,X,f_m(\theta)),\]
	for all $\theta \in \IBr(X)$. Hence there exists a fake $m$th Galois action on $\IBr(X)$ with respect to $X \lhd X \rtimes \Aut(X)$. 
\end{proof}

% ----------------------------------------------------------------
% ----------------------------------------------------------------
\section{Proof of Theorem A}\label{sec:maintheorem}
% ----------------------------------------------------------------
% ----------------------------------------------------------------

\begin{theoremB}
Let $S$ be a non-abelian simple group and let $X$ be the universal covering group of $S$. Then for all non-negative integers $m$ such that $(|X|, m ) = 1$, there exists a fake $m$th Galois action on $\IBr(X)$ with respect to $X \lhd X \rtimes \Aut(X)$.
\end{theoremB}

\begin{proof}

First suppose that $S$ has cyclic outer automorphism group. Then the result follows from \cite[Theorem 4.4]{N/S/T}. Now suppose that $S := \bG^F / \tn{Z}(\bG^F)$ is a simple group for some simple simply connected algebraic group $\bG$ and a Frobenius endomorphism $F$ as defined in Section~\ref{sec:fglt}. Then if $p = \ell$, the result follows from \cite[Theorem 5.1]{N/S/T}. If $p \neq \ell$ and $X = \bG^F$, then the result follows from Corollary~\ref{coro:typeA} and Proposition \ref{prop:nottypeA}. If $p \neq \ell$ and $X \neq \bG^F$, then the result follows from Propositions \ref{prop:except1} and \ref{prop:except2}.
\end{proof}

\bibliographystyle{acm}
\bibliography{BIBLIOGRAPHY}

\end{document}